\documentclass[11pt]{amsart}

\usepackage{comment}

\newtheorem{theorem}{Theorem}[section]

\newtheorem{corollary}[theorem]{Corollary}

\newtheorem{lemma}[theorem]{Lemma}

\theoremstyle{definition}
\newtheorem{remark}[theorem]{Remark}
\newtheorem{question}[theorem]{Question}

\usepackage[colorlinks, bookmarks=true]{hyperref}
\usepackage{color,graphicx,shortvrb}

\usepackage{enumerate}
\usepackage{amsmath}
\usepackage{amssymb}
\usepackage{tikz-cd}
\newcommand{\tri}[1]{| \! | \! | {#1} | \! | \! |}

\newcommand{\N}{\mathbb{N}}

\newcommand{\R}{\mathbb{R}}
\newcommand{\Q}{\mathbb{Q}}

\tikzcdset{scale cd/.style={every label/.append style={scale=#1},
		cells={nodes={scale=#1}}}}

\newcommand{\triple}[1]{{\left\vert\kern-0.25ex\left\vert\kern-0.25ex\left\vert #1 
    \right\vert\kern-0.25ex\right\vert\kern-0.25ex\right\vert}}

\begin{document}
	\numberwithin{equation}{section}
	
	% \title[Analogues of Krein-Milman]{Order extreme points, solid convex hull, and lattice analogues of Krein-Milman Theorem}
	
	\title{A Separable Universal Homogeneous Banach Lattice}
	
	\author{ M.A.~Tursi}
	
	\address{ Dept.~of Mathematics, University of Illinois, Urbana IL 61801, USA}
	\email{gramcko2@illinois.edu}
	
	\date{\today}
	
	%\subjclass[2010]{46B22, 46B42}
	
	%\keywords{Banach lattice, extreme point, convex hull, Radon-Nikod{\'y}m Property}
	
	%\dedicatory{To the memory of Victor Lomonosov}
	
	% \thanks{We wish to thank the organizers of Positivity IX, where part of this work was carried out.}
	
	% \date{}
	\maketitle
	
	\parindent=0pt
	\parskip=3pt
	
	\begin{abstract}
		We prove the existence of a separable approximately ultra-homogeneous Banach lattice $\mathfrak{BL}$ that is isometrically universal for separable Banach lattices.  This is done by showing that the class of Banach lattices has the Amalgamation Property, and thus finitely generated Banach lattices form a metric Fra\"iss\'e class. Some additional results about the structural properties of $\mathfrak{ BL}$ are also proven.
	\end{abstract}

	\maketitle
	\thispagestyle{empty}

\section{Introduction} 

This paper explores homogeneity in the class of separable Banach lattices.  We prove the existence of isometrically universal lattices for appropriate classes with varying levels of homogeneity and give some initial results about the structure of such spaces.  \\

Some comparable results are already known in Banach space theory.  The Gurarij space $\mathfrak G$ is an isometrically universal separable Banach space with the following homogeneity property:  for any finite dimensional spaces $A \subseteq B$, any isometric embedding $f:A \rightarrow \mathfrak G$, and any $C > 1$, there exists a map $g:B\rightarrow \mathfrak G$ extending $f$ such that $\frac{1}{C} \| x \| \leq \|g(x) \| \leq C \|x\|$ for all $x\in B$. Such separable Banach spaces are isometrically unique (see \cite{lusky}, as well as a simplified proof by Kubi\'s and Solecki in  \cite{kubis13}).  An alternate construction by Ben Yaacov characterizing $\mathfrak G$ as a metric Fra\"iss\'e limit is found in \cite{ben15}. As a Fraisse limit, $\mathfrak G$ has another kind of homogeneity that strengthens isomorphic embeddings with small distortion to isometric embeddings while sacrificing full commutativity.  In other words, given finite dimensional spaces $A\subseteq B$, an isometric embedding $f:A \rightarrow \mathfrak G$, and $\varepsilon > 0$, there exists an isometric embedding $g:B\rightarrow \mathfrak G$ such that $\|f - g|_A\| < \varepsilon$.  Since $\mathfrak G$ is a Fra\"iss\'e limit, it is also isometrically unique among separable spaces with this property. \\

We prove the lattice analogue of the above stated result.  Using Fra\"iss\'e machinery, we show that there is a unique isometrically universal separable Banach lattice $\mathfrak{ BL}$ with the following homogeneity property:  for any lattices $A\subseteq B$ generated by finitely many elements and lattice isometric embeddings $f:A\rightarrow \mathfrak{ BL}$, for all $(a_1,..., a_n) \subseteq A$ generating $A$, and for all $\varepsilon > 0$, there exists a lattice isometric embedding $g:B \rightarrow \mathfrak{BL} $ such that for each $a_i$, $\|f(a_i) - g(a_i) \|< \varepsilon $ (Theorem \ref{t:lattices-are-fraisse}).  The key to this result is the fact that Banach lattices have the Amalgamation Property (Theorems \ref{t:lattice-amalg} and \ref{t:arbitrary-lattice-amalg}). Observe that if $A$ and $B$ are finite dimensional, we can strengthen almost commutativity of the diagram restricted to generators to almost commutativity in norm.  $\mathfrak{ BL}$ can also be constructed as an inductive limit of $\ell_\infty^m(\ell_1^n)$ lattices, paralleling the construction of the Gurarij space as a limit of $\ell_\infty^n$ spaces (Theorem \ref{t:finitely-branchable-fraisse}).  \\

$\mathfrak{ BL}$ does not have the homogeneity property that originally characterized $\mathfrak G$, however, because in certain cases one cannot extend a lattice isometric embedding in a way that preserves both lattice structure and full commutativity in the separable setting.  In addition, even though $\mathfrak G$ can be "almost" homogeneous in either of the forms mentioned above, it cannot fully homogeneous in the sense of requiring both isometric embeddings and full commutativity of the diagram. Since it is unique, no separable spaces can have this stronger property. There exist non-separable Banach spaces, however, that are fully homogeneous, not just for the class of finite dimensional spaces, but also for separable spaces.  Such spaces, referred to as spaces of universal disposition, are constructed by Avil\'es, S\'anchez,  Castillo, and Moreno in \cite{avi11}.  A different construction (which assumes the CH) using Fra\"iss\'e sequences is given in  \cite{kubis}, where uniqueness is also established. Very recently, Avil\'es and Tradecete also constructed a (necessarily non-separable) lattice of universal disposition for separable lattices \cite{AvilesTradecete2020}. \\ 

Homogeneity in sublcasses of Banach lattices has been recently explored at length by Ferenczi, Lopez-Abad, Mbombo, and Todorcevic \cite{fer20}. This paper treats on various levels of homogeneity in $L_p$ Banach spaces, but it also explores lattice homogeneity. Specifically, for $1 \leq p < \infty$, the separable spaces $L_p(0,1)$ are Fra\"iss\'e limits for $\ell_p^n$ spaces with lattice embeddings as corresponding maps. The authors also construct an approximately ultra-homogeneous $M$-space for the class of finite dimensional $M$ spaces.  \\

Outside of the Banach lattice setting, homogeneous structures have been found for various classes.  Using injective objects, Lupini proved the existence of homogeneous structures for the classes  of function systems, $p$-multinormed spaces, and $M_q$-spaces \cite{lup16}. Certain $C^*$-algebras can also be constructed as Fra\"iss\'e limits of appropriate classes with relaxed conditions, including all UHF algebras, the hyperfinite II$_1$-factor \cite{Eagle16}, the Jiang-Su algebra \cite{masu16}, and more recently, a projectively universal AF-algebra constructed in \cite{ghas20}. 

\section{Preliminaries}\label{s:preliminaries}
We introduce definitions and notation that will be used in this paper.  For general information about Banach lattices, we refer the reader to \cite{mey,linden2}.  Throughout the paper, all Banach lattices are assumed to be real. \\

Given a Banach lattice $X$, we let $\mathbf B(X)$ and $\mathbf S(X) $ respectively refer to the unit ball and unit sphere of $X$. For $A \subseteq X$, let $A_+$ denote the elements of $A$ in the positive cone, and let $BL(A)$ be the Banach lattice generated by $A$. We say that $x\in A$ is an \textit{order extreme point} of $A$ if for all $y,z \in A$, if $x \leq ty - (1-t)z$ for $0 < t <1$, then $x = y = z$. A set $A\subseteq X$ is called \textit{solid} if for all $a\in X$ and $b\in A$, if $|a|:= a \vee (-a) \leq |b|$, then $a \in A$.  It turns out that for any solid set $A\in X$, a point $a \in A$ is an extreme point if and only if $|a|$ is order extreme \cite[Theorem 19.2]{OT19}. We also let $EP(A)$ be the set of  extreme points of $A$, let $OEP(A)$ be the set of order extreme points of $A$, and let $SCH(A) $ ($CSCH(A)$) be the (closed) solid convex hull of $A$. \\

The following is largely taken from \cite{ben15}: let $\mathcal L$ be a collection of symbols. These can be either \textit{predicate symbols} or \textit{function symbols}.  Each predicate or function symbol has an associated number called its \textit{arity}. We then call $\mathfrak A$ with associated metric space $A$ an \textit{$\mathcal L$-structure} if 

\begin{enumerate}
	\item For every predicate symbol $R$ with arity $n$, there is a continuous interpretation $R^\mathfrak{A}: A^n \rightarrow \mathbb R$.  We can also consider the distance to be a binary symbol (found in every structure).
	
	\item For every function symbol $f$ with arity $n$, we have a continuous interpretation $f^\mathfrak{A}: A^n \rightarrow A$.  Note that if a function symbol $c$ has $0$-arity, then it is a \textit{constant symbol}, and $c^\mathfrak{A} \in A$.
\end{enumerate}

These are different from the typical definitions of $\mathcal L$-structures in continuous logic as found in \cite{ben06}, the latter which require uniform continuity for functions and predicates but do not require that $X$ and $Y$ be bounded. The theory of Banach lattices can be formulated in the language $\mathcal L = (+, \R, \wedge, \vee)$. In particular, its function and predicate symbols have corresponding moduli of uniform continuity which are independent of their interpretation in a particular lattice. Given $ \overline{x} = (x_1,...,x_n)$, $ \overline{y} = (y_1,...,y_n) \subseteq X$, where $X$ is a metric space, we let  $$d( \overline{x},  \overline{y}) = \max_{i\leq n} d(x_i,y_i). $$ 
 
As in \cite[Chapter 2]{ben06}, we define the modulus of uniform continuity.  A function $\Delta_f: \R^+ \rightarrow (0,1]$ is a \textit{modulus of uniform continuity} for a $\mathcal L$-function or predicate symbol $f$ of arity $n$ if for all $\mathcal L$-structures $ \mathcal M$ and $\overline{x}, \overline{y} \in \mathcal M^n$, $d(\overline{x}, \overline{y}) < \Delta_f(\varepsilon) $ implies that $d(f^\mathcal{M}(\overline{x}), f^\mathcal{M}(\overline{y})) < \varepsilon$. For example, the function symbol $\wedge$ has modulus of continuity $\Delta (\varepsilon) = \frac14 \varepsilon $. That is, given a lattice $X$ and $(x_1, y_1)$ and $(x_2, y_2) \in X^2$, if $d((x_1,y_1), (x_2,y_2)) < \frac{1}{4}\varepsilon$, then $\| x_1 \wedge y_1 - x_2 \wedge y_2 \| < \varepsilon$. The definition of moduli of continuity in the appendix in \cite[Chapter 2]{ben06} assumes that moduli have domains restricted to $(0,1]$ but we can extend such functions to $\R^+$ by letting $\Delta_f(r) = \Delta_f(1)$ for all $r > 1$. Propositions  2.4 and 2.5 in \cite{ben06} show that compositions of uniformly continuous real functions and $\mathcal L$-function and $\mathcal L$-predicate symbols also have corresponding moduli of uniform continuity, since they are also uniformly continuous.\\

 We say that $A$ is a \textit{substructure} of $B$ if $A$ is a closed subset of $B$ which is also closed under all combinations of the function symbol operations.  For Banach lattices, $X$ is a substructure of $Y$ if it is a sublattice of $Y$. Let $f:A\rightarrow B$ be a map between two $\mathcal L$ structures.  If $f$ preserves norms, function operations, and predicate symbols in $\mathcal L$, then $f$ is considered a \textit{embedding}. \\

Let $\phi: X\rightarrow Y$ be a map between two Banach lattices.  We say that $\phi$ is a \textit{lattice homomorphism} if it is a bounded linear map that also preserves the lattice operations (i.e., $\phi(x\wedge y) = \phi(x) \wedge \phi(y)$).  To check whether a linear map is also a lattice homomorphism, by \cite[Theorem 1.34]{abram}, it is enough to check that it is positive ($x \geq 0 \implies \phi(x) \geq 0$) and preserves disjointness.  That is, if $x,y\in X$ and $x\perp y$ (i.e., $|x| \wedge |y| = 0$), then $\phi(x) \perp \phi(y)$.     For $C\geq 1$, we say that $\phi$ is a \textit{ $C$-embedding}, if for all $x\in X$, $\frac{1}{C} \|x\| \leq \| \phi(x) \| \leq C\|x\|$.  If $C = 1$, then $\phi$ is an embedding between Banach lattice structures, so we simply call it an embedding.  In subsequent sections, since this paper mainly deals with Banach lattices, we refer to lattice embeddings simply as embeddings. If a $C$-embedding is also surjective, then it is called an $C$-\textit{isometry}, and if $C = 1$, it is simply an isometry.  Observe that for any $C \geq 1$, if $\phi:X\rightarrow Y$ is a $C$-embedding, then $\phi(X)$ is a sublattice of $Y$, $\phi$ is a $C$-isometry from $X$ onto $\phi(X)$, and $\phi^{-1}$ is a $C$-isometry from $\phi(X)$ onto $X$. \\

 Let $A_0 \subseteq B$.  We then let $<A_0>$ be the substructure generated by $A_0$.  This can be understood as the smallest set $A \subseteq B$ with $A_0 \subseteq A$ and $A$ a substructure of $B$. We say that $A$ is \textit{finitely generated} if there exist $(a_1,...a_n) \subseteq A$ such that $A = <(a_1,...,a_n)>$.  Suppose that $\mathcal K$ is a class of finitely generated $\mathcal L$-structures.  If $A \in \mathcal K$, we say that $A$ is a \textit{$\mathcal{K}$- structure} if every finitely generated substructure of $A$ is also in $\mathcal K$. \\

 For $ A = < \overline{a} > $ and $ B = < \overline{b}>$ with $| \overline{a} | = | \overline{b}|$, we define 

\[ d^\mathcal{K} ( \overline{a},  \overline{b}) = \inf_{\substack{\phi_1:A\rightarrow C \\ \phi_2:B \rightarrow C}} d( \phi_1 (  \overline{a}), \phi_2 ( \overline{b}))  ,  \]

where $\phi_1$ and $\phi_2$ are both embeddings into some ambient $\mathcal K$-structure $C$. If we clearly understand generating tuples $ \overline{a}$ and $ \overline{b}$ for lattices $A$ and $B$ to be in some larger ambient space without necessary reference to explicit embeddings, we just write $d( \overline{a},  \overline{b})$ instead of $d( \phi_1 (  \overline{a}), \phi_2 ( \overline{b}))$. \\

Let $\mathcal K$ be a class of finitely generated structures.  We then say $\mathcal K$ is \textit{Fra\"iss\'e} if:
\begin{itemize}
	\item  $\mathcal K$ has the \textit{Hereditary Property} (HP): every member of $\mathcal K$ is a $\mathcal K$-structure.  
	
	\item $\mathcal K$ has the \textit{Joint Embedding Property} (JEP): any two $\mathcal K$-structures embed into a third.  (Note that if $\mathcal K$ has the JEP, then $d^\mathcal{K}$ is defined for all pairs of tuples in $\mathcal K$ of the same length.)
	\item $\mathcal K$ has the \textit{Near Amalgamation Property } (NAP): for any structures $ A =< \overline{a}>$, $ B_1$ and  $ B_2$ in $\mathcal K$ with embeddings $f_i:  A \rightarrow B_i$, and for all $\varepsilon > 0$, there exists a $ C \in \mathcal K$ and embeddings $g_i :  B_i \rightarrow  C$ such that $$d(g_1\circ f_1( \overline{a}), g_2\circ f_2 ( \overline{a}) ) < \varepsilon.$$ If  $g_1 \circ f_1 = g_2 \circ f_2$, then we just say that $\mathcal K$ has the Amalgamation Property (AP). Clearly the AP implies the NAP.

	\item  $\mathcal K$ has the \textit{Polish Property} (PP): if $\mathcal K$ has the JEP, HP and NAP, then $d^\mathcal{K} $ is a pseudo-metric over $\mathcal K$. If $d^\mathcal{K}$ is separable and complete in $\mathcal K_n$ (the $\mathcal K$-structures generated by $n$ many elements):
	
	\item $\mathcal K$ has the \textit{Continuity Property} (CP): every symbol in $\mathcal L$ is continuous on $\mathcal K$: that is, for function symbols, the map $( \overline{a}, \overline{b}) \mapsto ( \overline{a}, \overline{b}, f^{( \overline{a})}( \overline{a}))$ is continuous, and for predicate symbols $P$, the map $ \overline{a} \mapsto P^{ \overline{a}}( \overline{a})$ is continuous.
\end{itemize}

By \cite[Theorem 3.21]{ben15}, if $\mathcal K$ is Fra\"iss\'e, there exists a separable space $\mathfrak M$, known as the \textit{Fra\"iss\'e limit}, that is universal for $\mathcal K$  and \textit{approximately ultra-homogeneous} on $\mathcal K$.   That is, for all finitely generated structures $ A =< \overline{a}> \subseteq \mathfrak M$, embeddings $f:  A \rightarrow \mathfrak M$, and  $\varepsilon > 0$, there exists an automorphism $\phi: \mathfrak M \rightarrow \mathfrak M$ such that $d(f( \overline{a}), \phi( \overline{a})) < \varepsilon$.  Conversely, if a space $\mathfrak M$ is approximately ultra-homogeneous, its finitely generated substructures form a Fra\"iss\'e class, and $\mathfrak M$ is its limit.  Such a space is also isometrically universal for all separable $\mathcal K$ structures (including those which are not finitely generated).  \\

Instead of the PP and CP, a class $\mathcal K$ may have the following weakened conditions:

\begin{itemize}
	\item The \textit{Weak Polish Property} (WPP): the metric $d^{\mathcal K}$ is separable (but not necessarily complete)
	\item The \textit{Cauchy Continuity Property} (CCP): the map $( \overline{a}, \overline{b}) \mapsto ( \overline{a}, \overline{b}, f^{( \overline{a})}( \overline{a}))$ sends $d^\mathcal{K}$- Cauchy sequences to Cauchy sequences, and for predicate symbols $P$, the map $ \overline{a} \mapsto P^{ \overline{a}}( \overline{a})$ sends Cauchy sequences to Cauchy sequences,
\end{itemize}

If $\mathcal K$ has the HP, JEP, and NAP in addition to the two conditions above, then $\mathcal K$ is an \textit{incomplete Fra\"iss\'e class}.  A relevant example is that of finite dimensional $\ell_p$ spaces  (\cite[Section 4.2]{ben15} gives a brief discussion).  These have a (unique) Fra\"iss\'e limit of their completion, which is the class of separable $L_p$ spaces, and the limit is $L_p(0,1)$.  See \cite[Proposition 3.7]{fer20} for a recent proof using tools from analysis. \\

The notion of universal disposition in Banach space theory can also be studied with Banach lattices. Let $\mathcal C$ be a class of Banach lattices. A lattice $X$ is of \textit{approximately universal disposition} for a class $\mathcal C$ with lattices defined by finitely many elements if for all $A \in \mathcal C$ and for all embeddings $f:A\rightarrow X$, $g:A \rightarrow B$, with $A \in \mathcal{ C}$ defined by $ \overline{a}$ and $B\in \mathcal{  C}$, and for all $\varepsilon > 0$, there exists an $(1+\varepsilon)$-embedding $h:B \rightarrow X$ such that $\| h\circ g ( \overline{a}) - f( \overline{a}) \| < \varepsilon$.  Approximate universal disposition relaxes the condition in approximate ultra-homogeneity of the existence of an embedding down to a $(1+\varepsilon)$-embedding for arbitrarily small $\varepsilon$.\\

Definition by finitely many elements in lattices can occur in more than one way. One can speak, for example, of finite generation in the context of the logic of metric structures.  On the other hand, one might refer to finite dimensional lattices. For spaces of approximately universal disposition, if the class in question is finite dimensional lattices, we can let the finitely many atoms define the lattice's basis rather than generators doing so (in fact any finite dimensional lattice can be generated by two elements: see Theorem \ref{t:fb-implies-fg}), and we can strengthen the requirement that  $\| h\circ g ( \overline{a}) - f( \overline{a}) \| < \varepsilon$ to a norm requirement that $\| h \circ g - f \| < \varepsilon$. \\

Throughout the paper we rely on the notion of finite branchability. Let $E$ be a Banach lattice.  Let $(A_n)_n$ be a sequence of finite non-empty sets, and let $T = \cup_{k=0}^\infty \prod_{n=1}^{k} A_n$ be the tree generated by them. Suppose also that $(x_\sigma)_{\sigma \in T} \subseteq E_+$. We then say that $(x_\sigma)$ is a \textit{finitely branching tree} in $E_+$ if for all $\sigma$ with $|\sigma| = k$, $(x_{(\sigma^\frown b)})_{|\sigma| = k}$ is disjoint, and \[ x_\sigma = \sum_{b \in A_{k+1} } x_{(\sigma^\frown b)}. \]

Note that given the property outlined in the definition, $\text{span} (\{ (x_\sigma)_{\sigma \in T} \})$ is a vector lattice in $E$. If $\text{span} ( \{ (x_\sigma)_{\sigma \in T} \} )$ is dense in $E$ for some finitely branching tree $(x_\sigma)$, we call $E$ \textit{finitely branchable.}  Finitely  branchable lattices allow us to reduce problems involving finitely generated, but infinite dimensional lattices to that of finite dimensional lattices, since they are inductive limits of finite dimensional lattices.  It is also easy to show the other direction: If a lattice is the inductive limit of finite dimensional lattices, then it is finitely branchable. Finally, observe that finitely branchable lattices are separable.\\

Throughout, we will be working with two named classes of Banach lattices: let $\mathcal K$ be the class of finitely generated lattices, and $\mathcal K'$ be the class of sublattices of $\ell_\infty^m(\ell_1^n)$ spaces, with $m,n \in \N$. Let $\mathcal K_n \subseteq \mathcal K$ be the class of lattices generated by $n$ elements, and likewise for $\mathcal K'_n \subseteq \mathcal K' $.  Here we do not require that the generating elements be distinct or minimal.  We also will make use of the isometrically universal separable lattice $\mathcal U := C(\Delta, L_1[0,1])$ constructed in \cite{leungLi}. It turns out that $\mathcal U$ is finitely branchable and in particular is the inductive limit of an increasing union of lattices in $\mathcal K'$, which will be useful later on.  \\ 

We conclude this section with an outline of the rest of this paper. Section \ref{s:Amalg} explores the AP in Banach lattices and is split into two subsections.  In the first, we show show that any finite dimensional lattice can be approximated by a lattice in $\mathcal K'$ with arbitrarily small distortion (Lemma \ref{l:fd-into-1infty-l1}).  We use this result to prove an approximate amalgamation property for finite dimensional lattices (Theorem \ref{t:fd-approx-amalg}).  In particular, it is shown that $\mathcal K'$ has the AP. In the second subsection, we use the results in the first subsection to show that the class of Banach lattices has the AP (Theorems \ref{t:lattice-amalg} and \ref{t:arbitrary-lattice-amalg}).  The key to expanding the results on $\mathcal K'$ is the use of finitely branchable lattices.  We then end the section with some additional results on amalgamation over $C$-embeddings.  \\

In Section \ref{s:BL}, we prove the existence of a separable approximately ultra-homogeneous lattice $\mathfrak{ BL}$ by showing that $\mathcal K$ is a metric Fra\"iss\'e class and explore some of its structural properties (Theorem \ref{t:lattices-are-fraisse}).  The subclass $\mathcal K'$ is not just the first step to amalgamation; it is itself an incomplete Fra\"iss\'e class that is dense in the class of finitely generated separable lattices according to the Fra\"iss\'e metric (Lemma \ref{l:incomplete-fraisse}). We use this fact to show that $\mathfrak{ BL}$ is finitely branchable (Theorem \ref{t:finitely-branchable-fraisse}).  Finitely branchable lattices are themselves finitely generated (Theorem \ref{t:fb-implies-fg}), so unlike the Gurarij space, $\mathfrak{BL}$ is finitely generated, and in particular can be generated by two elements.\\

In Section \ref{s:sep-univ-dist}, we show that any separable lattice of approximately universal disposition for finitely generated lattices is isometric to $\mathfrak{BL}$ (Theorem \ref{t:1e-isometric-unique}).  We also construct lattices of approximately universal disposition for finite dimensional lattices and show that any such lattice which is also finitely branchable is isometric to $\mathfrak {BL}$ (Theorem \ref{t:finite-branch-unique}).  Finally, we show a self-similarity property of $\mathfrak{ BL}$: any non-trivial projection band in $\mathfrak{ BL}$ is isometric to $\mathfrak{ BL}$ (Theorem \ref{t:bands-are-also-BL}).

\section{Banach lattices and the Amalgamation Property} \label{s:Amalg}

The bulk of this section is dedicated to proving that the class of Banach lattices has the AP. As this paper was nearing its completion, Avil\'es and Tradecete independently proved that Banach lattices have the Amalgamation Property by generating pushouts using free Banach lattices (see \cite[Theorem 4.4]{AvilesTradecete2020}).  We give an alternative approach.  We first show that $\mathcal K'$ itself has the AP, and then expand this result to $\mathcal K$ and to lattices in general.

\subsection{The Amalgamation Property in $\mathcal K'$}
 We start with the following lemma:

\begin{lemma}\label{l:characterizing-K'}
Let $X$ be a finite dimensional lattice.  Then the following are equivalent:
\begin{enumerate}
	\item $OEP(\mathbf{B}(X))$ is finite.
	\item $EP(\mathbf{B}(X))$ is finite.
	\item $EP(\mathbf{B}(X^*))$ is finite.
\end{enumerate}
\end{lemma}
\begin{proof}
	(1) is equivalent to (2) by finite dimensionality and Theorem 19.2 in \cite{OT19}.  To show that (2) implies (3), suppose $\mathbf B(X)$ has finitely many extreme points.  Then by Theorem 16 in \cite{egg58}, it is the intersection of finitely many closed half-spaces.  Let $f_1,...,f_m \in \mathbf{S}(X^*)$ such that $\mathbf B(X) = \{ x\in X: f_i(x) \leq 1 \text{ for all }1\leq i \leq m\}. $  Then $\| x\| = \max f_i(x)$ for all $x\in X$, so $B(X^*) =  CH\{f_1,..., f_m\}$.  Otherwise, if $g\in \mathbf S(X^*) \backslash CH\{f_1,..., f_m\}$, by the Hahn-Banach separation theorem there exists some $x\in \mathbf S(X)$ such that \[ \sup_i f_i(x) < g(x). \]
	By Milman's theorem, all the extreme points of $\mathbf B(X^*)$ are contained in $\{f_1,..., f_m\}$, so $\mathbf B(X^*) $ has finitely many extreme points.  By reflexivity of finite dimensional lattices, (3) implies (2) as well.
\end{proof}
 
\begin{lemma}\label{l:fd-into-1infty-l1}
	Let $X$ be a finite dimensional Banach lattice.  Then for all $C > 1$, there exists a $C$-embedding from $X$ into an $\ell_\infty^m(\ell_1^M)$ space for some $m$, with $M = \dim X$.  If, furthermore, $X$ has finitely many order extreme points, then $X$ embeds isometrically into $\ell_\infty^m(\ell_1^M)$ space for some $m$.
\end{lemma}
\begin{proof}
	Suppose $\{x^*_1 ,..., x^*_m \} $ is an $\varepsilon$-net on $\mathbf S(X^*)_+$, where $\frac{1}{C} < 1 - \varepsilon$. Then for all $x\in \mathbf S(X)$, we have $ \frac{1}{C} < 1- \varepsilon \leq \sup_i x^*_i(|x|) \leq 1$.  Now  $X^*$ is also finite and is thus generated by its atoms, which are the evaluation functionals $e^*_i$ for the atoms $e_i \in X$, with $1\leq i \leq M$.  That is, if $x = \sum_j c_j e_j$, then $e^*_i(x) = c_i$.  These functionals form a basis in $X^*$, so we can assume $x^*_i = \sum_j a(i,j) e^*_j$, with $a(i,j) \geq 0$.  Based on this, consider the lattice  $\ell_\infty^m(\ell_1^M)$, and let $u(i,j) \in \ell_\infty^m(\ell_1^M)$ correspond to the $j$'th atom in the $i$'th copy of $\ell_1^M$.  Then let $\phi(e_j) = \sum_i a(i,j) u(i,j)$.  \\
	
	$\phi$ is a lattice homomorphism, since it is a positive linear map that maps atoms to disjoint elements.   It also is a $C$-embedding. Indeed, let $x = \sum c_j e_j \in \mathbf S(X)_+$ . Then $$\phi(x) = \sum_j c_j \phi (e_j) = \sum_j \sum_i c_j a(i,j) u(i,j),$$
	
	so $$\| \phi(x) \| = \sup_i \sum_j |c_j|a(i,j) = \sup_i \sum_j a(i,j) e^*_j (x) = \sup_i x^*_i(x).$$
	
	Thus $\frac{1}{C} \|x\| \leq \|\phi(x) \| \leq \|x\|$.\\
	
	If $\mathbf{B}(X)$ has finitely many order extreme points, then by Lemma \ref{l:characterizing-K'}, so does the dual unit ball $\mathbf{B}(X^*)$.  Let $\{x^*_1,..., x^*_m \} = OEP(\mathbf{B}(X^*))$. Then $\|x \| = \sup_{1\leq i \leq m} x^*_i(|x|)$.  Construct $\phi$ in the same way as above, and observe that $\| \phi(x) \| = \sup_i x^*_i (x) = \| x \|$.
\end{proof}

Lattices in $\mathcal K'$ play a key role in subsequent results on homogeneous lattices and their structure.  We present some of the notation that will be used in subsequent proofs:  Suppose $F\in \mathcal K'$, and let $(e_1,...,e_m)$ be the atoms of $F$.  Let $f:F \rightarrow G:= \ell_\infty^N (\ell_1^M)$ be a $C$-embedding, with $C \geq 1$. Let $u(k,j)$ be $j$'th atom in the $k$'th copy of $\ell_1^M$. We then have, for each $e_i \in F$, that $f(e_i) = \sum_{k,j} a^i(k,j) u(k,j)$.  Note that $f$ maps atoms to disjoint positive elements, so we can just let $a^i(k,j) = a(k,j)$, and sum up only over atoms that support $f(e_i)$.  Specifically, we fix a row $k$ and let
$$F^k_{i} = \{ j \leq M: f (e_i) \wedge u(k,j) > 0 \}. $$

Then $$f(e_i) = \sum_k \sum_{j\in F_i^k} a(k,j) u(k,j).$$

Observe that $F$ and $f$ induce an $N \times m$ matrix $A^f_F$, with $A^f_F(k,i) = \sum_{j\in F_i^k} a(k,j)$. If $F_i^k$ is empty, then $A_F^f(k,i) = 0$.   It turns out the rows of $A_F^f$ capture $F$'s structure completely, while small distortions in $f$ imply small distortions in $A_F^f$.  We give a lemma to this effect. From now on, if we have two $C$-isometries  $f_j: F \rightarrow G_j$ with $j = 1,2$ be $C$-isometries with $C\geq 1$, with $G_1$ and  $G_2$ both $\ell_\infty^N (\ell_1^M)$ spaces, we just let $A = A_F^{f_1}$ and $B = A_F^{f_2}$.  For $1 \leq l \leq N$, we also let $A(l) = (A(l,1), A(l,2),..., A(l,m)) $ and  $B(l) = (B(l,1), B(l,2),..., B(l,m)) $.

\begin{lemma} \label{l: perturbations-in-unitball}
	Let $f_j: F \rightarrow G_j$ with $j = 1,2$ be $C$-isometries with $C\geq 1$, and suppose $G_1$ and $G_2$ be $\ell_\infty^N (\ell_1^M)$ spaces. Then for all rows $l$, we have $A(l) \in C^2SCH(\{B(k): 1\leq k \leq N\})$.  In particular, if each $f_j$ is an embedding, then $SCH(\{B(k): 1\leq k \leq N\}) = SCH(\{A(k): 1\leq k \leq N\})$.
\end{lemma}
\begin{proof}	
 Let $r\in \mathbf{B}(\ell_\infty^M)_+$. Then there exists some row $k$ such that for all rows $l$, $\frac{1}{C}\sum r_n B(l,n) \leq  \|\sum r_ne_n \| \leq C\sum r_n A(k,n)$. Now that $ B(l) \notin C^2 SCH(\{A(k): k \leq N\})$ for some $l$.  Then by \cite[Proposition 19.7]{OT19} there exists some $r \in \mathbf{B}(\ell_\infty^M)_+$ such that \[ C^2 \sup_{y \in SCH({A(k)})} \sum r_n y_n <  \sum r_n B(l,n) \leq C^2 \sum r_n  A(k,n) \] 
for some $k$, which is a contradiction. 
\end{proof}

In the case that $C = 1$, recall that the construction in Lemma \ref{l:fd-into-1infty-l1} used $N$ rows of $\ell_1^M$ to correspond to the $N$ order extreme points in the unit ball of $X^*$.  Since we can think of the rows in $A_F^f$ as elements in the dual space $F^*$, where $A(l) (\sum c_i e_i) = \sum c_i A(l,i)$, then by Lemma \ref{l: perturbations-in-unitball}, we actually have $SCH(\{A(l)\}) = \mathbf B(F^*)$.  In particular, any $F\in \mathcal K'$ has finitely many order extreme points.  Combined with Lemma \ref{l:characterizing-K'}, we thus have the following result:

\begin{corollary}\label{c:4-equvalences}
 The four following properties are equivalent for finite dimensional lattices $X$:	
	\begin{enumerate}
	\item $OEP(\mathbf{B}(X))$ is finite.
	\item $EP(\mathbf{B}(X))$ is finite.
	\item $EP(\mathbf{B}(X^*))$ is finite.
	\item $X \in \mathcal K'$.
	\end{enumerate}	
\end{corollary}

We now prove the following:

\begin{theorem}\label{t:fd-approx-amalg}
Suppose for $j=1,2$, $f_j: E\rightarrow F_j$ are $C$-embeddings with $F_1$ and $F_2$ in $\mathcal K'$ with $C\geq 1$. Then there exist  $G\in \mathcal K'$ and  $C$-embeddings $g_j:F_j \rightarrow G$ such that $g_1 \circ f_1 = g_2 \circ f_2$. That is, the following diagram commutes:
\begin{center}
	\begin{tikzcd}		
	E \arrow{r}{f_1} \arrow{d}{f_2} & \arrow{d}{g_1} F \\
	F_2 \arrow{r}{g_2} & G		
	\end{tikzcd}
\end{center}
 In particular,  $\mathcal K'$ has the AP.
\end{theorem}
\begin{proof} 
	We can assume that $F_j = \ell_\infty^N (\ell_1^{M_j})$, where $M_j = \dim F_j$ for $j = 1,2$. Let $e_1,...,e_n$ be the atoms in $E$, and for $F_1$ and $F_2$, we let $u(k,j)$ and $v(k,j)$, respectively, correspond to the $j$'th atom in the $k$'th row (that is, the $k$'th copy of $\ell_1^{M_j}$).  For row $l$ and atom $i$, we let $$F^l_{1,i} = \{ j \leq M_1: f_1 (e_i) \wedge u(l,j) > 0 \}, $$ and similarly, we let $$F_{2,i}^l = \{ j \leq M_2:  f_2 (e_i) \wedge v(l,j) > 0 \}$$ 
	
	We now define $g_1$ and $g_2$.  Let $F'_j$ be the the lattice ideal in $F_j$ generated by $f_j(E)$, and let $ (F'_1\otimes F'_2) \oplus F_1\oplus F_2$ be understood as a vector lattice with atoms of the form $u(k,j) \otimes v(l,m)$, $u(k,j)$, and $v(l,m)$.   For $u(k,j)$ with $j \in F_{1,k}^i$, let $g_1(u(k,j)) = u(k,j)\otimes   f_2(e_i)$.  If $u(k,j)\notin F_{1,k}^i$ for any $i$, let $g_1(u(k,j)) = u(k,j)$.  For $v(l,m)\in F_{2,l}^i$, let $g_2(v(l,m)) =  f_1(e_i) \otimes v(l,m)$, and if $v(l,m)\notin F_{2,l}^i$ for any $i$, let $g_2(v(l,m)) = v(l,m)$.  First, we show that $g_1  \circ f_1 = g_2 \circ f_2$.  Indeed, we have 
	\begin{align*}
	g_1 \circ f_1\bigg(\sum_i c_i e_i\bigg) & =  \sum_i c_i g_1\bigg(\sum_k \sum_{j\in F_{1,k}^i} a(k,j) u(k,j) \bigg) \\
	& = \sum_i c_i \bigg( \sum_k \sum_{j\in F_{1,k}^i} a(k,j) u(k,j) \bigg) \otimes  f_2 (e_i) \\
	& = \sum c_i (f_1(e_i) \otimes  f_2 (e_i)), 	
	\end{align*}
	and similarly:
	\begin{align*}
	g_2 \circ f_2\bigg(\sum_i c_i e_i\bigg) & =  \sum_i c_i g_2\bigg(\sum_l \sum_{m\in F_{2,l}^i} b(l,m) v(l,m) \bigg) \\
	& = \sum_i c_i   f_1 (e_i) \otimes \bigg(\sum_l \sum_{m\in F_{2,l}^i} b(l,m) v(l,m) \bigg)  \\
	& = \sum c_i  (f_1(e_i)  \otimes  f_2 (e_i)). 
	\end{align*}
	
Let $G= BL ( g_1(F_1) \cup g_2(F_2))$, and let the unit ball of $G$  be $SCH\big(g_1(\mathbf{B}(F_1)) \cup g_2(\mathbf{B}(F_2)) \big)$. \\
	
Note that $g_1$ and $g_2$ are both contractive maps. We now show that they are also $C^2$-embeddings.  This will be sufficient, because then we can replace $g_1$ and $g_2$ with $Cg_1$ and $Cg_2$ while still preserving commutativity in the diagram.  These latter maps are themselves $C$-embeddings, thus proving the theorem. Since $g_i(\mathbf{B}(F_i))$ has only finitely many order extreme points, and since the resulting space is finite dimensional, $SCH\big(g_1(\mathbf{B}(F_1)) \cup g_2(\mathbf{B}(F_2)) \big)$ is also closed.  So we need only to show without loss of generality that if $x, y\in \mathbf S(F_1)_+$, $z\in \mathbf S(F_2)_+$, and $r   g_1(x) \leq tg_1(y) +(1-t)g_2(z)$ with $0\leq t \leq 1$, then $r   < C^2$.\\
	
	Suppose $x,y$ and $z$ are as above. Since for some $k \leq N$, $\big\| x \wedge \big( \sum_j u(k,j) \big) \big\| = 1$ and $g_1\big(x \wedge(\sum_j u(k,j))\big) \leq g_1(x)$, we can  assume that $x = \sum_j c_j u(k,j)$ with $\sum c_j = 1$. Furthermore, since 
	$$r   g_1(x) \leq tg_1(y) \wedge r   g_1(x) +(1-t)g_2(z) \wedge r   g_1(x), $$
	we can also assume that $ tg_1(y) \leq r  g_1(x) $, so $y = \sum d_j u(k,j)$, with $\sum d_j \leq 1$.  Finally, we can let $z = \sum_1^M \mu_n z_n$, where $z_n$ is an order extreme point in $F_2$; that is, there is a sequence $s^n =(s^n_l)_l$ of length $N$ such that $z_n = \sum_l v(l, s^n_l)$, and furthermore, $\mu_n > 0$ with $\sum \mu_n = 1$.  Then
	$$ 0 \leq g_1 (r  x - ty)  \leq g_2 \bigg( \sum_i \sum_l \sum_{n: s_l^n = m \in F_{2,l}^i}  \mu_n v(l, m) \bigg). $$

	Now both sides of the inequality are supported, and the left hand side fully supported, by atoms of the form $u(k,j)\otimes v(l,m)$ where  $j\in F^i_{1,k}$ and $m\in F^i_{2,l}$.   Thus, for any $u(k,j) \in {F'}^\perp_1$, we must have  $r  c_j -td_j = 0$, since $g_1({F'}^\perp_1)$ is disjoint from $g_2(F_2)$, and similarly $g_2({F'}^\perp_2)$ is disjoint from $g_1(F_1)$. Therefore 
	\[ r  x - ty =  \sum _i \sum _{j \in F_{1,k}^i} ( r  c_j - td_j ) u(k,j) \]
	
	Recall that for each coefficient the left hand side must be less than or equal to the right hand side.  Evaluating both sides, we thus have that
	\begin{align*} &\sum _i \bigg(\sum _{j \in F_{1,k}^i} ( r  c_j - td_j ) u(k,j)\bigg) \otimes f_2 (e_i) \\ 
	=  &\sum _i \bigg[\bigg(\sum _{j \in F_{1,k}^i} ( r  c_j - td_j ) u(k,j)\bigg) \otimes \bigg( \sum_l  \sum_{m\in F_{2,i}^l} b(l,m) v(l,m) \bigg)\bigg] \\
	= & \sum_i \bigg[ \sum_{j\in F_{1,k}^i} \sum_l \sum_{ m\in F_{2,i}^l }  (r  c_j - td_j)b(l,m) \ u(k,j) \otimes v(l,m) \bigg] \\
	\leq &(1-t)\sum_i \bigg[\bigg( f_1(e_i) \wedge \sum_j u(k,j)\bigg) \otimes \bigg(  \sum_l \sum_{n: s_l^n = m \in F_{2,l}^i}  \mu_n v(l, m)\bigg)\bigg] \\
	= &\sum_i (1-t)\bigg[\bigg( \sum_{j \in F_{1,k}^i} a(k,j)u(k,j) \bigg) \otimes \bigg(  \sum_l \sum_{n: s_l^n = m \in F_{2,l}^i}  \mu_n v(l, m)\bigg)\bigg]\\
	= & \sum _i \bigg[ \sum_{j\in F_{1, k}^i} \sum_l \sum_{m\in F_{2,l}^i} (1-t)a(k,j) \bigg(\sum_{n:s^n_l = m} \mu_n\bigg) u(k,j) \otimes v(l,m) \bigg]
	\end{align*}
	For each $i$, for all $j\in F_{1, k}^i$, for each $l$, and for all $m\in F_{2,l}^i$, the coefficient of $u(k,j) \otimes v(l,m)$ on the left hand side is $(r  c_j - td_j)b(l,m)$, and on the right hand side, we have $\displaystyle (1-t)a(k,j) \sum_{n: s_l^n = m} \mu_n$.  Thus 
	$$ (r  c_j - td_j)b(l,m) \leq (1-t)a(k,j) \sum_{n: s_l^n = m} \mu_n.$$
	
	Let $A = A_{E}^{f_1}$ and $B = A_{E}^{f_2}$, as defined prior to Lemma \ref{l: perturbations-in-unitball}.	Adding across all $m \in F_{2,l}^i$, we have:
	\begin{align*} (r  c_j - td_j) \sum_{m\in F_{2,l}^i} b(l,m) \leq (1-t)a(k,j) \sum_{m\in F_{2,l}^i} \sum_{n: s_l^n = m } \mu_n 
	\end{align*}
	
so $	(r  c_j - td_j) B(l,i) \leq (1-t)a(k,j) \lambda_l^i,$ where $$\lambda_l^i = \sum_{m\in F_{2,l}^i} \sum_{n: s_l^n = m } \mu_n.$$  Observe that $ \sum_i \lambda_l^i = 1$ for all rows $l$.  Add up terms over all $j \in F_{1,k}^i$.  Thus
	\begin{align*}
	\sum_{j\in F_{1,k}^i} (r  c_j - td_j) B(l,i) & \leq \sum_{j\in F_{1,k}^i} (1-t)a(k,j) \lambda_l^i \implies \\
	(r  C_i - tD_i) B(l,i) & \leq (1-t)A(k,i) \lambda_l^i,
	\end{align*}
	where $C_i = \sum_{j \in F_{1,k}^i} c_j$ and $D_i = \sum_{j\in F_{1,k}^i} d_j$. Now if $$C' = \sum_{u(k,j) \in {F'}^\perp_1} c_j \quad \text{and} \quad   D' = \sum_{u(k,j) \in {F'}^\perp_1} d_j, $$ then $\sum_i C_i +C' = 1$ and $\sum_i D_i +D' \leq 1$. Since $r  x - ty \geq 0$, it follows that $r  C_i - tD_i \geq 0$ and since for each $j$, $u(k,j) \in {F'}^\perp_1$ implies $rc_j -t_jd_j =0$, we have $ r  C' - tD' = 0$.  By Lemma \ref{l: perturbations-in-unitball}, there exists a finite sequence $(\nu_l)_{l=1}^{N}$ such that $A(k) \leq C^2\sum \nu_l B(l)$, with $\sum_l \nu_l = 1$ and $\nu_l \geq 0$.  Then in particular, 
	$$ (r  C_i - tD_i) A(k,i)  \leq C^2 (1-t)A(k,i) \sum_l \nu_l \lambda_l^i $$

	If $A(k,i) = 0$, then $C_i$ and $D_i$ are also 0, since $F_{1,k}^{i}$ is empty.  Otherwise $A(k,i) > 0$, so  for all $i$,
	\begin{align*}
	(r  C_i - tD_i) & \leq C^2 (1-t) \sum_l \nu_l \lambda_l^i \implies \\
	\sum_i (r  C_i - tD_i)+ r  C' - tD' & \leq  C^2 (1-t) \sum_l \sum_i \nu_l \lambda_l^i  \implies \\
	r - t \leq r- t(\sum_i D_i+D')& \leq  C^2 (1-t) \sum_l \nu_l (\sum_i \lambda_l^i) \implies \\
	r - t  & \leq C^2 (1-t) \implies \\
	r & \leq C^2
	\end{align*}
 Thus $g_1$ (and by similar argument $g_2$) is a $C^2$-embedding. \\
 
 Finally, $G$ itself has finitely many order extreme points, so by Lemma \ref{l:fd-into-1infty-l1} it can be embedded into a $\ell_\infty^m(\ell_1^n)$ space, implying that $G \in \mathcal K'$. 
 \end{proof}
	
\begin{corollary}\label{c:fd-approx-amalg}
	Let $E, F_1, F_2$ be finite dimensional lattices, let $C \geq 1$, and suppose $f_1:E\rightarrow F_1$ and $f_2:E\rightarrow F_2$ are $C$- embeddings.  Then for all $\varepsilon > 0$, there exist a lattice $G \in \mathcal K'$ and $(C+ \varepsilon)$-embeddings $g_1:F_1 \rightarrow G$ and $g_2: F_2\rightarrow G$ such that $g_1\circ f_1 = g_2 \circ f_2$.  
\end{corollary}

\begin{proof}	
  Pick $\delta$ such that $(1+\delta)^2C < C+\varepsilon$, and pick $N$ such that there are $(1+\delta)$-embeddings $\phi_j:F_j \rightarrow F'_j := \ell_\infty^N(\ell_1^{\dim F_j})$.  Then each $\phi_j \circ f_j : E \rightarrow F'_j$ is a $C(1 + \delta)$-embedding.  By Theorem \ref{t:fd-approx-amalg}, there exists $G \in \mathcal K'$ and $C(1+\delta)$-embeddings $g'_j:F'_j \rightarrow G$ for $j\in  \{1,2\}$ such that $g'_1 \circ \phi_1 \circ f_1 = g'_2 \circ \phi_2 \circ f_2$.  Now let $g_j = g'_j \circ \phi_j$, and observe that each $g_j$ is a $(1+\delta)^2C$-embedding, and $g_1\circ f_1 = g_2\circ f_2$.  Since $(1+\delta)^2C <C +\varepsilon$, we are done.	
\end{proof}

\subsection{The Amalgamation Property for arbitrary Banach lattices}
The above approach works well with finite dimensional lattices, but expanding to finitely generated lattices will lead to some additional complications since finitely generated lattices need not be finite dimensional.  In fact, the separable isometrically universal lattice $\mathcal U = \mathcal C(\Delta, L_1(0,1))$ can be generated by two elements (see Remark 3.1 in \cite{leungLi}). However, we can use this result to express separable lattices with sequences of finite dimensional lattices in order to demonstrate a general amalgamation.\\

Suppose now that $E$ is a Banach lattice.  Let $\alpha$ be a limit ordinal, and let $(E_\gamma)_{\gamma < \alpha}$ be a sequence of increasing sublattices of $E$ such that $\overline{\cup_{\gamma< \alpha} E_\gamma} = E$.  considering $(\gamma)_{\gamma < \alpha}$ as a net, define $\mathcal E \subseteq \prod E_\gamma$ by

$$ \mathcal E = \{ (x_\gamma)_{\alpha < \gamma}: \lim\limits_\gamma x_\alpha = x \in E \}.  $$

Essentially, $\mathcal E$ is a lattice of $\alpha$-length sequences converging to elements in $E$, with norm $\| (x_\alpha) \|_\mathcal{E} = \sup_\alpha \|x_\alpha \|$.

\begin{lemma}\label{l:lat-to-sequence}
	Let $E$ and  $\mathcal E$ be as above, and let $\mathcal E_0$ be the ideal in $\mathcal E$ of null sequences.  Then $E$ is isometric to $\mathcal E/ \mathcal E_0$.
\end{lemma}
\begin{proof}
	Let $ x\in E$, and let $(x_\gamma)_\gamma \rightarrow x$, and let $[(x_\gamma)]$ denote the equivalence class induced by $\mathcal E_0$.  \\ We will now show that the map $ g:E \rightarrow \mathcal E/\mathcal E_0$ with  $g(x) \mapsto [(x_\gamma)_\gamma] $ is an isometry.\\
	
	First, it is well defined:  if $(x_\gamma)_\gamma $ and $(y_\gamma)_\gamma$  converge to $x$, then $(y_\gamma -x_\gamma)_\gamma  \in \mathcal E_0$, so $[(x_\gamma)_\gamma] = [(y_\gamma)_\gamma]$.   By continuity of scalar multiplication and addition, $g(x)$ is linear.  It also preserves norms.  Note that $\| g(x) \|_{\mathcal E/ \mathcal E_0} = \inf \{\| (x_\gamma)_\gamma \|: (x_\gamma)_\gamma \rightarrow x \}$, so $\|g(x) \|_{\mathcal E/\mathcal E_0} \geq \|x\|$, since $\|x_\gamma \| \rightarrow \|x\|$.  For $\epsilon > 0$ and $(x_\gamma) \rightarrow x$, there exists some $\beta < \alpha$ such that for all $\gamma > \beta, \|x_\gamma - x \| < \epsilon$  consider then the $\alpha$-sequence $(x'_\gamma)$ with $x'_\gamma = 0$ for all $\gamma < \beta$ and $x'_\gamma = x_\gamma$ otherwise.  Then $g(x) =[ (x'_\gamma)_\gamma ]$, and so $\| g(x) \|_{\mathcal E/\mathcal E_0} \leq \|x\| +\varepsilon$. In addition, the map is clearly surjective, since any $[(y_\gamma)] = g(x)$ where $(y_\gamma) \rightarrow x$. Thus $g$ is a linear isometry.\\
	
	Finally, $g$ preserves lattice operations. First of all, $g$ is positive.  If $x \geq 0$, and $(x_\gamma)_\gamma \rightarrow x$, then $ 0 \leq (x_\gamma \vee 0)_\gamma \rightarrow x\vee 0 = x$.  Since $\mathcal E_0$ is a lattice ideal, $g(x) =[ (x_\gamma \vee 0)_\gamma] = [(x_\gamma)_\gamma] \vee [0] \geq 0$.  In addition, $g(x)$ preserves disjointness:  if $x \wedge y = 0$, then if $x_\gamma \rightarrow x$ and $y_\gamma \rightarrow y$, then $x_\gamma \wedge y_\gamma \rightarrow x\wedge y = 0$. Then $$[(x_\gamma)]\wedge [(y_\gamma)]  = [(x_\gamma \wedge y_\gamma)]  = [0],$$ so $g$ is a lattice homomorphism. Therefore $g$ is a lattice isometry. \end{proof}

Given a separable lattice $E$, by  \cite[Proposition 2.2]{leungLi}, there exists a finitely branchable lattice $E'$ such that $E \subseteq E' \subseteq E^{**}$. Let $(x_\sigma)_{T_{E'}}$ be the corresponding branching tree.    Let $k_n \uparrow \infty$ where $ k_n \in \N$ be a strictly increasing sequence, let ${E'}_{k_n} = \text{span}({x_\sigma: |\sigma| = k_n})$, and let $\mathcal E' \subseteq \prod_n E'_{k_n} $ be the lattice defined by 
\[ \mathcal E' = \{ (x_i)_i :  x_i \rightarrow x \in E'  \},  \]

with lattice norm $\|x\|_{\mathcal E'} = \sup_n \|x_n\|$. Finally, let $\mathcal E'_0 = \{ x \in \mathcal E': x_n \rightarrow 0 \}$.  By Lemma \ref{l:lat-to-sequence},  $\mathcal E' / \mathcal E_0' $ is lattice isometric to $E'$ itself.  Furthermore, any finite dimensional lattice $F \in E$ can be approximated by a sublattice of some $E_n$ for some $n$:  

\begin{lemma}\label{l:finite-lattice-approximation}
	Let $E = \overline{ \cup_n E_n}$ where $(E_n)$ is an increasing sequence of lattices.  Suppose $F \subseteq E$ is a finite dimensional sublattice.  Then for all $\varepsilon > 0$ there exist $n\in \N$ and a $(1+\varepsilon)$-isometry $g:F \rightarrow E_n$ such that $\| g - Id|_F \| <\varepsilon$.
\end{lemma}

\begin{proof}
	
	Let $m= \dim F$, and let $h(x_1,...,x_m ) = x_1 - x_1\wedge (\bigvee_{i \geq 2} x_i)$. Let $(e_i)_i$ be the atoms of $F$, and let $\overline{e_i} = (e_1,..., e_{i-1}, e_{i+1},..., e_m)$.  Now $h$ is continuous, and for any sequence $(x_1,...,x_m)$ of positive elements, the elements $h(x_1, \overline{x_1}),...,h(x_m, \overline{x_m})$ are mutually disjoint and positive. Thus since $\cup E_n$ is dense in $X$, for some $n$ there exist corresponding positive $(f_1,...,f_m) \subseteq E_{n}$ such that $\| h(e_i, \overline{e_i}) - h(f_i, \overline{f_i}) \| < \delta/m$.  Now since the $e_i$'s are mutually disjoint, $h(e_i, \overline{e_i}) = e_i$.  Let $g:F \rightarrow E_n$ be the lattice homomorphism generated by $g(e_i) = h(f_i, \overline{f_i})$.  Then for any $\sum a_i e_i \in \mathbf S(F)$, we have $$ \| \sum e_i - \sum a_i g(e_i) \| \leq \sum_{i}^m |a_i|\| e_i - g(e_i) \| < \delta.$$  It follows that $ 1-\delta < \| g(\sum_i^m a_i e_i)  \| < 1+\delta$, so $g$ is a $\frac{1+\delta}{1-\delta}$-isometry. If we let $\frac{1+\delta}{1-\delta} < 1+\varepsilon$, we have both that $g$ is a $(1+\varepsilon)$-isometry and $\| Id|_F - g\| < \varepsilon$.
	
\end{proof}
We now state the following lemma:
\begin{lemma}\label{l:quotient-lift}
	Let $E$ and $A$ be a finitely branchable Banach lattices, and suppose $\phi:E \rightarrow A$ is an embedding. Let $(x_\sigma)_{\sigma \in T_E}$ and $(y_\sigma)_{\sigma \in T_A}$ be linearly dense spanning trees for $E$ and $A$, respectively. Then for all $\varepsilon > 0$, there exist a strictly increasing sequence $(k_n)_n \subseteq N$ and $(1+\varepsilon)$-embedding $\phi':\mathcal E \rightarrow \mathcal A$ generated by a sequence of maps $\phi_n: E_n \rightarrow A_{k_n}$ such that: \begin{enumerate}
		\item The following diagram commutes:
		
		\begin{center}
			\begin{tikzcd}
			\mathcal E \arrow{r}{\phi'}  \arrow{d}{q_E} & \mathcal A \arrow{d}{q_A} \\
			E \arrow{r}{\phi} & A
			\end{tikzcd}
			%\caption{}
		\end{center}
		
		\item For each $n$,  $\phi_n:E_n\rightarrow A_{k_n}$ is a $(1+\varepsilon/2^n)$-embedding. 
	\end{enumerate}
\end{lemma}

\begin{proof}
	Let $\mathcal E \subseteq \prod_n E_n$.  We will construct $k_n$ as follows. Begin with $x_\emptyset \in E_+$, and suppose that $\|x_\emptyset\| = 1$.  Pick $k_0 \in \N$ and $z_\emptyset \in \text{span} ( \{ y_\sigma: |\sigma| = k_0 \} ) $ with $z_\emptyset\geq 0$ such that $\| z - \phi(x_\emptyset) \| < \varepsilon$.  We then let $\phi'_0(x_\emptyset) = z_\emptyset$.  For $n>0$, since $E_n$ is finite dimensional and embeds into $A$, by Lemma \ref{l:finite-lattice-approximation}, pick $k_n$ in such a way that such a way that there is a $\phi_n:E_n \rightarrow A_{k_n}$ with distortion level at most $(1+ \varepsilon/2^n)$. \\
	
	Let $\phi' = (\phi_n)_n$.  Note that $\phi'$ sends atoms to disjoint elements and is a positive linear map.  To show that property 1 is also fulfilled, we must first show that $\phi$ takes elements in $\mathcal E$ to elements in $\mathcal A$.  Let $x \in \mathcal E$, with $(x_i) \rightarrow x' \in E$. Now $\phi(x_i) \in A$, and by continuity $\phi(x_i) \rightarrow \phi(x') \in A$ as well.  Yet $\| \phi(x_i) - \phi_i(x_i) \| \leq \frac{\varepsilon}{2^i}$, so $\phi_i(x_i) \rightarrow \phi(x') \in \mathcal A$.   In addition,  if $x_i \rightarrow x \in E$, then $q_A \circ \phi' ((x_i)_i) = \phi(x)$, which gives us commutativity, thus fulfilling property 1.    
\end{proof}
We are now ready to prove the following:
\begin{theorem}\label{t:lattice-amalg}
	Let $E,A_1, A_2$ be separable Banach lattices, and let $f_1: E\rightarrow A_1$ and $f_2:E\rightarrow A_2$ be embeddings.  Then there exists a separable Banach lattice $G$ and embeddings $g_1: A_1 \rightarrow G$ and $g_2: A_2\rightarrow G$ such that $g_1 \circ f_1 = g_2 \circ f_2$.
		%\item If $E,A_1$ and $A_2$ are finitely generated, we can make $G$ finitely generated as well.
\end{theorem}
\begin{proof}
	Since each $f_i:E\rightarrow A_i$ is a lattice embedding for $i=1,2$, by \cite[Theorem 1.4.19]{mey}, each $f_i^{**}: E^{**} \rightarrow A_i^{**}$ is a lattice embedding.  By Proposition 2.2 in \cite{leungLi}, there exists a separable finitely branchable lattice $E \subset E' \subseteq E^{**}$ with a finitely branching tree $(x_\sigma)_{\sigma \in T_{E'}}$.  Similarly, we can take the Banach lattice generated by $f_i^{**}(E')$ and $A_i$, and inject it into a finitely branchable $A'_i$ with a corresponding finite branching tree $(y_\sigma)_{\sigma \in T_{A'_i}}$.   Thus we can redefine $f_1$ and $f_2$ to be extended to $E'$.  \\
	
	Let $\varepsilon > 0$, and using Lemma \ref{l:quotient-lift}, pick appropriate increasing sequences of natural numbers $k^1_n \uparrow \infty$ and $k^2_n \uparrow \infty$ generating $\mathcal A'_1$ and $\mathcal A_2'$ with accompanying $(1+\varepsilon)$-isometries $f'_1$ and $f'_2$ such that the following diagram commutes: 

	
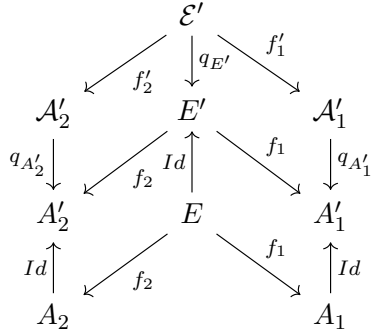
\begin{center}
		
		\begin{tikzcd}
		\qquad & \mathcal E' \arrow{rd}{f'_1} \arrow{d}{q_{E'}} \arrow{ld}{f'_2} & \qquad \\
		\mathcal A_2' \arrow{d}[left]{q_{A_2'}} & E' \arrow{ld}{f_2} \arrow{rd}{f_1} &  \mathcal A_1' \arrow{d}{q_{A_1'}}  \\
		A_2' & E \arrow{u}{Id}  \arrow{rd}{f_1} \arrow{ld}{f_2} & A_1' \\
		A_2 \arrow{u}{Id} &  & A_1 \arrow{u}[right]{Id}
		\end{tikzcd}
	\end{center}

	%Given the fact that $A'_{j\ m} \subseteq A'_{j \ n}$ whenever $m < n$, any of the maps $\phi_n^j$ can have any $A'_{j \ m}$ as a codomain so long as $m > k^j_n$.  Thus for ease of notation, inductively construct single increasing sequence $(k_n)$, with $k_n > k^1_n, k^2_n$ satisfies the conditions for both $\mathcal A_1$ and $\mathcal A_2$.   \\
	
	By the assumptions on Lemma \ref{l:quotient-lift}, $f'_j = (\phi_n^j)_n$, where $\phi_n^j : E'_n \rightarrow A'_{j \ k^j_n}$ is a $(1+\varepsilon/2^n)$-isometry.  Use Corollary \ref{c:fd-approx-amalg} to get $G_n$ and $(1+\varepsilon/{2^{n-1}})$-embeddings $\psi^1_n$ and $\psi^2_n$ such that $\psi^1_n \circ \phi^1_n = \psi^2_n \circ \phi^2_n$,  and let $g'_1 = (\psi^1_n)_n$ and $g'_2 = (\psi^2_n)_n$. Let $\mathcal G' \subseteq \prod_n G_n$ be the sublattice generated by $g'_1(\mathcal A'_1)$ and $g'_2(\mathcal A'_2)$, and equip $\mathcal G'$ with the sup-norm; that is, if $x\in \mathcal G' $, let $\| x\|_{\mathcal G'} = \sup \| x_n \|_{G_n}$. Now each $g'_j$ is a $(1+2\varepsilon)$-embedding.  Let $\mathcal G'_0$ be the ideal consisting of elements $x \in \mathcal G'$ such that $\|x_n\|_{G_n} \rightarrow 0$, and let $G = \mathcal G'/\mathcal G'_0$. Note that for each $j\in \{1,2\} $, we have $g'_j (\mathcal A'_{j \ 0}) \subseteq \mathcal G'_0$. Thus $g'_j$ induces well defined maps $g_j: A'_j \rightarrow G$, with $g_j = q_{G} \circ g'_j \circ q_{A'_j}^{-1}$.  We therefore have the following commuting diagram: 
	
	
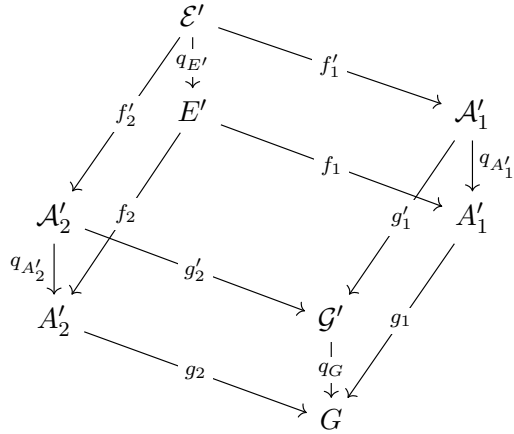
\begin{center}
		\begin{tikzcd}
		\qquad & \mathcal E' \arrow{rrd}[description]{f'_1} \arrow{d}[description]{q_{E'}} \arrow{ldd}[description]{f'_2} & \qquad \\
		& E' \arrow{ldd}[description]{f_2} \arrow{rrd}[description]{f_1} & &  \mathcal A_1' \arrow{d}{q_{A_1'}}  \arrow{ldd}[description]{g'_1} \\
		\mathcal A_2' \arrow{d}[left]{q_{A_2'}} \arrow{rrd}[description]{g'_2}&  & & A_1'  \arrow{ldd}[description]{g_1}\\		
		A_2' \arrow{rrd}[description]{g_2} & &\mathcal G' \arrow{d}[description]{q_{ G}}&  & \\
		& & G
		\end{tikzcd}
	\end{center}
	
	It remains to show that each $g_j$ is in fact an embedding.  To this end, we note that if $z \in G$, then $\| z \| = \inf  \{ \|y \| : q_G (y) = z \}$.  Let $x \in A'_1.$ Pick $y\in \mathcal A'_1$ with $\| y \| < 1 + \delta$ such that $y_i \rightarrow x$.  This can be done by picking $n$ such that for all $n \geq N$ $\|x - y_n \| < \delta$, $\varepsilon/2^{n-1} < \delta$, and furthermore, we can assume that for all $n < N$, $y_n = 0$.  It then follows that $ \frac{1}{(1+\delta)^2} \leq \| g'_1(y) \|_{\mathcal G'} \leq (1+\delta)^2$, so $\| q_G  g'_1 (y) \| \leq (1+\delta)^2$, Thus $ \|g_1(x) \|_{G'} \leq (1+\delta)^2$.  In addition, for any $z \in \mathcal G'_0$, since for all $\delta' > 0$ we have $\|z_n \|_{G_n} \leq \delta'$ for all large enough $n$, it follows that $\|z - g'_1(y) \| > \frac{1}{(1+\delta)^2} - \delta'$.  Thus $\|g_1(x) \|_{G} \geq \frac{1}{(1+\delta)^2}$.  $\delta$ can be chosen to be arbitrarily small, so $\|g_1(x) \|_{G}  = 1$. \\
	
	Finally, we show that $g_j$ preserves disjointness and is a positive map.  Let $x\in A_{j+}'$, and chose a sequence $y = (y_i)_i \in \mathcal A'_{j+}$ with $y_i \rightarrow x$. Then $g'_j (y) \geq 0$, so $q_G g'_j (y) = g_j(x) \geq 0$.  To show preservation of disjointness, let $x, x' \geq 0$ be disjoint elements, and let $y = (y_i)_i \in q_{\mathcal A'_j}^{-1}(x)$ and similarly let $y'= (y'_i)_i \in q_{\mathcal A'_j}^{-1}(x')$.  Then $y \wedge y' \in \mathcal A'_{j \ 0}$; since $(y_i)_i \rightarrow x$ and $(y'_i)_i \rightarrow x'$, we have $$y \wedge y' = (y_i \wedge y'_i)_i \rightarrow x\wedge x' = 0,$$ so $g'_j(y) \wedge g'_j(y') = g'_j(y \wedge y') \in \mathcal G'_0$, which means that $g_j(x)\wedge g_j(x') = q_Gg'_j(y) \wedge  q_G g'_j(y') = q_G g'_j(y \wedge  y') = 0$. Thus $g_j$ is an embedding. \\
	
	To show separability, we simply restrict $g_1 $ and $g_2$ to $A_1$ and $A_2$, and replace $G$ with the lattice generated by $g_1(A_1)\bigcup g_2(A_2)$.  Thus if $A_1$ and $A_2$ are both separable, then so is $G$.	\end{proof}

\begin{remark}
	We can also ensure that $G$ is finitely generated, since we can embed $G$ into $\mathcal U$ if necessary. Thus $\mathcal K$ has the AP.
\end{remark}

We can expand Theorem \ref{t:lattice-amalg} for arbitrary lattices with a similar proof.

\begin{theorem}\label{t:arbitrary-lattice-amalg}
	Let $E, F_1, F_2$ be Banach lattices, and let $f_i:E \rightarrow F_i$, with $i\in \{ 1,2\}$ be embeddings.  Then there exists a lattice $G$ and isometric embeddings $g_i:F_i \rightarrow G$ such that $g_1 \circ f_1 = g_2 \circ f_2$.  Furthermore, if $F_i$ has density character no more than $\kappa$, we can ensure that $G$ does as well.
\end{theorem}

\begin{proof}
	We prove this by ordinal induction over the density character $\kappa$.  For the base case of $\kappa = \aleph_0$, this was already proven in Theorem \ref{t:lattice-amalg}.  Suppose now that we have shown the same for all lattices of density character less than $ \kappa$.   Let  $(z_\gamma)_{\gamma <\kappa}$ be a $\kappa$-sequence dense in $E$, and let $(x^i_\alpha)_{\alpha < \kappa}$ be $\kappa$-length sequences dense in $F_i$.  Let $E^\beta = BL( (z_\alpha)_{\alpha<\beta})$ and let $F_i^\beta = BL((x^i_\alpha)_{\alpha < \beta} \cup f_i(E^\beta))$.  Then $E^\beta \uparrow E$, $F^\beta \uparrow F$, and $f_i(E^\beta) \subseteq F^\beta_i$. Now each $f_i$ induces an embedding $\phi_i:\mathcal E \rightarrow \mathcal F_i$, where $\phi_i( (y_\beta)_{\beta< \kappa}) = (f_i(y_\beta))_{\beta< \kappa}$. \\
	
	 Both $E^\beta$ and the $F^\beta_i$'s have dense subsets of size strictly less than $\kappa$, so by induction, pick $G^\beta$ and embeddings $\psi_i^\beta: F_i^\beta \rightarrow G^\beta$ such that $ \psi_1^\beta \circ f_1|_{E^\beta} =  \psi_2^\beta \circ f_2|_{E^\beta}$. Let $\psi_i = (\psi_i^\beta)_{\beta < \alpha}$, and let $\mathcal G$ be the sublattice of $\prod_{\beta} G^\beta$ generated by the elements of $\psi_i(\mathcal F_i)$.  Let $\mathcal G_0$ be the ideal in $\mathcal G$ of nets converging in norm to $0$, and let $G = \mathcal G/\mathcal G_0$.  Now let $g_i = q_G \circ \psi_i \circ q_E^{-1}$.  Use the same argument as in Theorem \ref{t:lattice-amalg} to show that each $g_i$ is well defined, an embedding, and together with $G$ give the desired amalgamation.  Finally, $G$ has the desired density character if we restrict it to the lattice generated by $g_1(F_1) \cup g_2(F_2)$.
\end{proof}

We end this section with some additional results on the interplay between the AP and $C$-embeddings.  In each of these cases, we can perturb lattices or maps that change $C$-embeddings into embeddings in exchange for full commutativity or preservation of the original norm:

\begin{theorem}\label{t:isomorphism-to-isometry}
	Let $f:A \rightarrow X$ be a $C$- embedding.  Then there exists a $C$-equivalent renorming $\triple{\cdot}$ of $X$ such that $f:A \rightarrow (X, \triple{\cdot})$ is an embedding. Furthermore, 
	
	\begin{itemize}
		\item 	if $f$ is an expansion (that is, if $f^{-1}$ is contractive), then we can make $\triple{\cdot} \leq \|{\cdot} \|$.
		\item 	if $f$ is a contraction,  then we can make $\triple{\cdot} \geq \| \cdot \|$.
		\item  if $A$ and $X$ are both in $\mathcal K'$, then we can ensure that $(X,\triple{\cdot})$ is also in $\mathcal K'$.
	\end{itemize}

\end{theorem}
\begin{proof}
	
	We start with a proof for the case when $f$ is an expansion. Let $\mathbf{B'} = CSCH \big(f(\mathbf{B}(A)) \cup \mathbf{B}(X) \big)$  be the unit ball of $\triple{\cdot}$. Observe that $\mathbf{B'} \supseteq \mathbf{B}(X)$ and $f(\mathbf{B}(A)) \subseteq C\mathbf{B}(X)$, so $ \frac{1}{C} \| \cdot \| \leq  \triple{\cdot} \leq \| \cdot \|$. \\
	
	We now show that $f:A\rightarrow (X, \triple{\cdot})$ is an embedding. Suppose that there exist $z_n \leq t_nf(x_n) + (1-t_n)y_n$ with $(1+\alpha)f(x) = \lim_n z_n$, with $x, x_n, y_n \geq 0$, $\alpha \geq 0$, $\|x_n\|, \|y_n\| \leq 1$, $0 \leq t_n \leq 1$, and $\|x\| = 1$.  By compactness, we can suppose $t_n$ converges to $t$, and just let $z_n \leq tf(x_n) + (1-t)y_n$.  Furthermore, we can assume that $\|f((1+\alpha )x -tx_n)\|\rightarrow b_x$.  Then for all $n$, we have $$f((1+\alpha)x - tx_n) \leq  (1-t)y_n + \delta_n$$ with $\|\delta_n\| \rightarrow 0$.  We then have 
	\[ 1+\alpha - t \leq \|f( (1+\alpha)x - tx_n) \| \leq (1-t) +\|\delta_n\|. \]
	Thus $1+\alpha - t \leq b_x \leq 1-t$, so $\alpha = 0$.\\
	
	For contractive $f$, let $ \mathbf{B'} $ be the closed solid convex hull of $f(\mathbf{B}(A)) \cup \frac{1}{C} \mathbf{B}(X)$.  Note here that $\frac{1}{C}\mathbf{B}(X) \subseteq \mathbf{B'} \subseteq \mathbf{B}(X)$, so $\| \cdot \| \leq \triple{\cdot} \leq C\|\cdot\|$.  Then use the same type of argument.\\
	
	For the general case, observe that $Cf$ is an expansion which is also a $C^2$-embedding.  Then by the proof of the first case, there is $C^2$-equivalent renorming $\triple{\cdot} \leq \| \cdot \|$ of X with $Cf:A\rightarrow (X, \triple{\cdot})$ an embedding. Now take the new norm of $X$ to be $C \triple{\cdot}$.  Then $f:A \rightarrow (X, C\triple{\cdot})$ is an embedding, and $C\triple{\cdot}$ is $C$-equivalent to $\| \cdot \|$.\\
	
	Finally, if $A, X \in \mathcal K'$, the unit ball of the renormed lattice $(X,\triple{\cdot})$ has finitely many order extreme points, so by Corollary \ref{c:4-equvalences}, $(X,\triple{\cdot})$ is also in $\mathcal K'$.
\end{proof}
Theorem \ref{t:isomorphism-to-isometry} can be used to generalize Theorem \ref{t:arbitrary-lattice-amalg} to diagrams involving $C$-isometries:

\begin{corollary}\label{c:morph-to-met-amalg}

Let $f_i:E \rightarrow F_i$ with $i = 1,2$ be $C_i$-embeddings for lattices $E$, $F_1$, and $F_2$.  Then:
\begin{itemize}
	\item  There exist a lattice $G$ and $C_i$-embeddings $g_i:F_i \rightarrow G$ such that $g_1 \circ f_1 = g_2 \circ f_2$. 
	\item There exist a lattice $G$, an embedding $g_1:F_1 \rightarrow G$, and a $C_1C_2$-embedding $g_2:F_2\rightarrow G$ such that $g_1 \circ f_1 = g_2 \circ f_2$. 
	\item If $E,F_1$, and $F_2$ are in $\mathcal K'$, we can ensure $G\in \mathcal K'$ as well.
\end{itemize}
\end{corollary}
	\begin{proof}
 For the first part, let $F'_i = (F_i, \triple{\cdot})$ be $C_i$-equivalent renormings such that $f_i$ is an embedding into $F'_i$. By Theorem \ref{t:arbitrary-lattice-amalg} (Theorem \ref{t:fd-approx-amalg}), there exists $G$ and embeddings $g_i:F'_i \rightarrow G$ such that $g_1 \circ f_1 = g_2 \circ f_2$. Since $F'_i$ is $C_i$-equivalent to $F_i$, each $g_i$ is a $C_i$-embedding on $F_i$. For the second part, use Theorem \ref{t:isomorphism-to-isometry} to simply renorm $G$ with a $C_1$-equivalent norm $\triple{\cdot}$ so that $g_1:F_1 \rightarrow (G, \triple{\cdot})$ is now an embedding.  Then $g_2:F_2 \rightarrow (G,\triple{\cdot})$ is a $C_1C_2$-embedding. For both parts, $G$ can be in $\mathcal K'$ if $E, F_1,$ and $F_2$ are in $\mathcal K'$.
\end{proof}

\begin{theorem}\label{t:two-embeddings-one-isomorphism}
	Suppose $f: X\rightarrow Y$ is a $(1+\varepsilon)$-embedding, and suppose $X, Y$ are in $\mathcal K$ (or $\mathcal K'$).  Then there exists a lattice $Z \in \mathcal K$ ($\mathcal K'$) and  embeddings $g:X \rightarrow Z$ and $h:Y\rightarrow Z$ such that $\| g - h\circ f \| \leq \varepsilon $.
\end{theorem}
\begin{proof}
	Let $j_1:X\rightarrow X \oplus_\infty f(X)$, with $j_1(x) = x\oplus \frac{1}{1+\varepsilon}f(x)$.  Let $j_2: f(X) \rightarrow X \oplus_\infty f(X)$ with $j_2(f(x)) = \frac{1}{1+\varepsilon} x \oplus f(x)$.  Note then that since $\frac{1}{1+\varepsilon}\| f(x) \| \leq \|x\| \leq (1+\varepsilon) \|f(x) \|$, $j_1$ and $j_2$ are both embeddings.  Then
	\begin{align*}
	\| j_1(x) - j_2 f(x)\| = & \bigg\| \bigg(1-\frac{1}{1+\varepsilon}\bigg)x \oplus \bigg(\frac{1}{1+\varepsilon} - 1\bigg)f(x) \bigg\|\\ = & \frac{\varepsilon}{1+\varepsilon}\|x \oplus -f(x) \|\leq \varepsilon \|x \|.
	\end{align*}
	If $f$ is surjective, then let $g=j_1$ and $h = j_2$, and we are done.  Otherwise,  $f(X) \subseteq Y$ and $j_2:f(X) \rightarrow X \oplus_\infty f(X)$ in an embedding, so use Theorem \ref{t:arbitrary-lattice-amalg} (or Theorem \ref{t:fd-approx-amalg}) to get a lattice $Z$ in $\mathcal K$ (respectively $\mathcal K'$)  and  embeddings $h_1: Y\rightarrow Z$ and $h_2:X\oplus_\infty f(X) \rightarrow Z$ such that $h_1|_{f(X)} = h_2 \circ j_2$. Then for all $x \in X$, $$\|h_2j_1(x) - h_1f (x) \| = \|h_2 j_1(x) - h_2j_2f(x) \| = \| j_1(x) - h_2(f(x)) \| \leq \varepsilon \|x\|.$$
	Let $g = h_2 \circ j_1$ and $h = h_1$, and we are done. \end{proof}

\begin{corollary}
	Let $E, F_1,F_2$ be lattices in $\mathcal K$ (or $\mathcal K'$), and let $f_j:E\rightarrow F_j$ be $(1+\varepsilon)$-embeddings.  Then there exist $H\in \mathcal K$ ($\mathcal K'$) and embeddings $g_j:F_j \rightarrow G$ such that $\|g_1 \circ f_1 - g_2\circ f_2\| \leq 2\varepsilon$.
\end{corollary}
\begin{proof}
	By Theorem \ref{t:two-embeddings-one-isomorphism} there exist $F'_j \in \mathcal K$ ($\mathcal K'$) and embeddings $f'_j:E \rightarrow F'_j$ and $\phi_j:F_j\rightarrow F'_j$ such that $\|f'_j - \phi_j \circ f_j \| \leq \varepsilon$.  Now use Theorem \ref{t:arbitrary-lattice-amalg} (or  Theorem \ref{t:fd-approx-amalg}) to get $H\in \mathcal K$ ($\mathcal K'$) and embeddings $g'_j:F'_j \rightarrow H$ with $g'_1\circ f'_1 = g'_2 \circ f'_2$. Let $g_j = g'_j \circ \phi_j$.  Then
	\begin{align*} 
	\|g_1 \circ f_1 - g_2 \circ f_2 \| = & \|g'_1\circ \phi_1 \circ f_1 - g'_2\circ \phi_2 \circ f_2 \|\\
	 \leq &\| g'_1\circ( \phi_1 \circ f_1 - f'_1) \| +  \| g'_2\circ( \phi_2 \circ f_2 - f'_2) \| \leq 2\varepsilon.  \end{align*}
\end{proof}

\section{The approximately ultra-homogeneous separable lattice $\mathfrak{BL }$}\label{s:BL}

The main result of this section is:
\begin{theorem}\label{t:lattices-are-fraisse}
	The class $\mathcal{K}$ of finitely generated separable Banach lattices is a Fra\"iss\'e class. Thus there exists a separable approximately ultra-homogeneous Banach lattice $\mathfrak {BL}$.
\end{theorem}

The level of homogeneity in $\mathfrak {BL}$ cannot significantly be strengthened. For one,  $\mathfrak{ BL}$ can only be "approximately" ultra-homogeneous, since no lattice automorphism can map non-weak units to weak units. $\mathfrak{BL}$ is clearly also atomless.  If there were an atom $e$ in $\mathfrak{BL}$, any automorphism would have to map it to another atom. Thus two embeddings $g_1:\R \rightarrow <e> \subseteq \mathfrak{BL}$ and  $g_2: \R \rightarrow <x> \subseteq \mathfrak {BL}$ such that $x$ both is disjoint from $e$ and not an atom cannot be arbitrarily approximated by an automorphism. \\

$\mathfrak {BL}$ is isometrically universal for separable Banach lattices, but it is not isometric to $\mathcal U$ because the latter is not approximately ultra-homogeneous and Fra\"iss\'e limits are unique up to isometry.  Indeed, let $<e>$ be a one-dimensional lattice generated by $e$, let $f_1(e) = a:= \vec{1}_\Delta \otimes \chi_{[0,1]}$,  $f_2(e) = b:= \vec{1}_K \otimes \chi_{[0,1]}$, where $K \subseteq \Delta$ is a proper clopen subset, and let $b' = a-b$.  Let $\phi$ be any automorphism over $\mathcal U$.  Now the sets $K_b=\{k\in \Delta: \| \phi(b)(k) \|_1 = 1\}$ and $K_{b'}=\{k\in \Delta: \| \phi(b')(k) \|_1 = 1\}$ are non-empty, and furthermore $b(K_{b'}) = 0$, and vice versa, since $\phi(<b,b'>)$ is isometric to $\ell_\infty^2$.  It follows that for $k\in K_{b'}$, we have $\|\phi(b)(k) -a(k) \|_1 = 1$, so $\| \phi(b) -a\| \geq 1$.\\

\begin{proof}[Proof of theorem]
It is clear that $\mathcal K $ has the HP and the JEP. It also has the CP by virtue of the fact that each function symbol in the language of Banach lattices has a fixed modulus of continuity independent of its interpretation. By Theorem \ref{t:lattice-amalg}, it has the AP.  It remains to show that it has the PP.  We need to show that the class of finitely generated Banach lattices is both separable and complete under the metric $d^{\mathcal K}$.   For separability, let $(x^n)_n$ be a countable dense subset of $\mathcal U$.  Then the set $\{ <x_{i_1},...,x_{i_n} > \}$ of lattices generated by finitely many elements in $(x_n)_n$ is itself a countable dense subset of $\mathcal K_n$.  \\

To show completeness, we use Theorem \ref{t:lattice-amalg} and the fact that Banach lattices are closed under direct limits.  Let $( \overline{a_i})_i$ be a Cauchy sequence of tuples generating structures in $\mathcal K_n$.  By passing to a subsequence if necessary, we can assume that $d^{\mathcal K}( \overline{a_{i}},  \overline{a_{i+1}}) < \frac{1}{2^{i+1}}$. For $ \overline{a_i}$ and $ \overline{a_{i+1}}$, let $B_i^1$ be a finitely generated lattice containing isometric copies of $< \overline{a_i}>$ and $< \overline{a_{i+1}} >$ such that $d( \overline{a_i},  \overline{a_{i+1}}) < d^\mathcal{K}( \overline{a_i},  \overline{a_{i+1}}) + \frac{1}{2^{i+1}}$.  Note then for each $i$, we have embeddings $< \overline{a_{i+1}} > \rightarrow B^1_i, B^1_{i+1}$, so use amalgamation to embed $B^1_i$ and $B^1_{i+1}$ into some finitely generated space $B^2_i$ where the associated diagram commutes.  Proceed inductively in a similar manner: each $B^k_{i+1}$ injects into $B^{k+1}_i$ and $B^{k+1}_{i+1}$, so use amalgamation to inject them into some finitely generated $B^{k+2}_i$.  The resulting commutative diagram illustrates the process:

\begin{center}
	\begin{tikzcd}
	< \overline{a_1}> \arrow{r} & B_1^1 \arrow{r}  & B^2_1 \arrow{r}  &B^3_1 \arrow{r} & \dots \\
	< \overline{a_2} > \arrow{r} \arrow{ru} & B_2^1 \arrow{r} \arrow {ru} & B^2_2 \arrow {ru} \\
	< \overline{a_3} > \arrow{r} \arrow{ru} & B_3^1  \arrow {ru} \\	
	\vdots \arrow{ru} 
	\end{tikzcd}	
\end{center}

Let $X$ be the closed inductive limit of the sequence of lattices $(B^n_1)_n$.  $X$ is itself separable, though it need not be finitely generated.  It also contains an isometric copy of each $< \overline{a_{i}}>$ and for each $ \overline{a_{i}},  \overline{a_{j}} \subseteq X$ with $i \leq j$, we have $$d( \overline{a_{i}},  \overline{a_{j}}) < \sum_{k=i}^{j-1} \bigg( d^{\mathcal K}(\overline{a_k}, \overline{a_{k+1}}) + \frac{1}{2^{k+1}} \bigg)  < \sum_{k=i}^{j-1} 2^{-k} < 2^{-i+1.}$$ Thus $( \overline{a_{i}})_i$, as a sequence of tuples in $X$, is Cauchy.  Let $ \overline{a} = \lim_i  \overline{a_{i}}$.  Since $X$ is complete, the sublattice $< \overline{a} >$ exists, which implies the completion of the metric $d^\mathcal{K}$. Thus $\mathcal K$ has the PP, and we are done.
\end{proof}

We continue with an additional characterization of $\mathfrak{BL}$.  In particular, $\mathfrak{ BL}$ is finitely branchable and finitely generated. To this end, we concentrate on the sub-class $\mathcal K'$.  \\

 The $\ell_\infty^m(\ell_1^n)$ lattices are in certain ways analogues of $\ell_\infty^n$ spaces. For one, recall the definition of an injective Banach space $E$: If $T:F\rightarrow E$ is a linear map and $F$ is a subpace of $G$, then there exists a linear map $\hat{T}:G\rightarrow E$ extending $T$ such that $\| T \| = \|\hat{T} \|$. There is also a lattice analogue of injectivity: We say $E$ is an injective lattice if for all lattices $F\subseteq G$ and any positive linear maps $T:F\rightarrow E$, then there exists a positive linear map $\hat{T}:G\rightarrow E$ extending $T$ such that $\| T \| = \|\hat{T} \|$.  The injective finite dimensional Banach spaces are exactly the $\ell_\infty^n$ spaces.  By \cite[Theorem 5.2]{cart75}, the $\ell_\infty$-sums of finite dimensional $\ell_1$ spaces make up the collection of finite dimensional injective lattices.  Furthermore, the Gurarij space in particular can be constructed as an inductive limit of $\ell_\infty^n$ Banach spaces.  We will now also show that $\mathfrak{BL}$ can be constructed as an inductive limit of $\ell_\infty^m(\ell_1^n)$ lattices.

\begin{lemma}\label{l:incomplete-fraisse}
	 $\mathcal K'$ is an incomplete Fra\"iss\'e class that is dense in $\mathcal K$.  In particular, the Fra\"iss\'e metric $d^{\mathcal K'}$ isometrically coincides with $d^{\mathcal K}$.  
\end{lemma}

\begin{proof}
	
$\mathcal K'$ has the HP by its definition. By Theorem \ref{t:fd-approx-amalg}, it has the AP.  Clearly it also has the JEP: given two  $X,Y \in \mathcal K'$, we also have $A\oplus_\infty B \in \mathcal K'$. To show density in $\mathcal K$,  let $< \overline{a}>$ be a finitely generated lattice, and embed $< \overline{a}>$ into $\mathcal U$.  Let $(x_\sigma)_{\sigma \in T}$ be the finitely branching tree comprised of elements in $\mathcal U$ of the form $\chi_{N_\sigma} \otimes \chi_{Q_k}$ where $Q_k$ is a diadic interval of length $2^{-n}$, $|\sigma| = n$, and $N_\sigma \subseteq \Delta = \{ 0, 1\}^\N $ is the set consisting of all infinite branches starting with $\sigma$. Now $\text{span}((x_\sigma)_{\sigma \in T})$ is dense in $\mathcal U$.  Let $S_n:= \{x_\sigma: |\sigma| = n \}$, and observe that $ \text{span} \{ S_n \}$ is itself a $\ell_\infty^{2^n}(\ell_1^{2^n})$ space. Given $\varepsilon > 0$, choose $n$ and $ \overline{x} \subseteq \text{span} (S_n)$ such that $d( \overline{a},  \overline{x}) < \varepsilon$. Then the lattice $< \overline{x} > \in \mathcal K'$ is sufficiently close to $<\overline{a} >$ in $d^{\mathcal K}$. \\

To show separability of $\mathcal K'$ and the CCP, it is sufficient to show that $d^\mathcal{K}|_{\mathcal K'} = d^{\mathcal K'}$. Clearly $d^\mathcal{K}|_{\mathcal K'} \leq d^{\mathcal K'}$, so we need only to show the opposite inequality. Let $d^{\mathcal K}( \overline{a},  \overline{b} ) = \delta$, and let $\varepsilon > 0$.  Choose embeddings $\phi_A$ and $\phi_B$ from $< \overline{a}>$ and $< \overline{b}>$ into $\mathcal U$ such that $d( \phi_A(\overline{a}), \phi_B( \overline{b})) < \delta + \varepsilon$.   By Lemma \ref{l:finite-lattice-approximation}, given $\varepsilon > 0$, there exist $n\in \N$ and  $(1+\varepsilon)$-embeddings $f_A:A\rightarrow D:=\text{span}(S_n)$ and $f_B:B\rightarrow D$ such that $\|f_A - \phi_A\| < \varepsilon$ and $\|f_B - \phi_B\| < \varepsilon$.  Now note that $d(f_A(\overline{a}), f_B(\overline{b})) < \delta+ 3\varepsilon$. Use Theorem \ref{t:isomorphism-to-isometry} to renorm $D$ with a $(1+\varepsilon)$-equivalent renorming $\triple{\cdot}$ so that $f_B:B\rightarrow D' = (D, \tri{\cdot}) \in \mathcal K'$ is an embedding. Then $f_A$ is a $(1+\varepsilon)^2$-embedding into $D'$.  Finally, use Theorem \ref{t:two-embeddings-one-isomorphism} to get some $C \in \mathcal K' $ and embeddings $g_A: A \rightarrow C$ and $g_{D'}:D' \rightarrow C$ such that $\|g_{D'} \circ f_A - g_A \|\leq 2\varepsilon + \varepsilon^2$.  Then for each $i$, we have
\begin{align*}
\| g_A(a_i)  - g_{D'}f_B (b_i) \|_C & \leq \| g_A(a_i) - g_{D'}f_A(a_i) \|_C + \| g_{D'}f_A(a_i)  - g_{D'}f_B (b_i) \|_C\\
 &\leq 2\varepsilon + \varepsilon^2 + \|f_A(a_i) - f_B(b_i) \|_{D'} \\
 &\leq 2\varepsilon +\varepsilon^2 + (1+\varepsilon)d(f_A(\overline{a}), f_B(\overline{b})) \\
 & \leq 2\varepsilon +\varepsilon^2+ (1+\varepsilon)(\delta + 3\varepsilon) .
\end{align*}

We can let $\varepsilon$  get arbitrarily small, so $d^{\mathcal{K}'}  \leq d^\mathcal{K}|_{\mathcal K'}$, and we are done.
\end{proof}
 It is known that incomplete Fra\"iss\'e classes admit a Fra\"iss\'e limit for the completion of the class, but here we will explicitly show that the construction of the limit $\mathfrak{BL}$ need only involve an increasing sequence of lattices in $\mathcal K'$.  In order to prove the following theorem, we use approximate isometries as described in \cite{ben15}, that is, bi-Katetov maps $\psi: X\times Y \rightarrow [0,\infty]$ with $X$ and $Y$ both metric spaces.   Recall that $\psi$ is bi-Katetov if for all $x,x_0\in X$ and $y,y_0\in Y$, $|\psi(x,y) - d(x,x_0)| \leq \psi(x_0,y)$ and  $|\psi(x,y) - d(y,y_0)| \leq \psi(x,y_0)$. In this context, approximate isometries provide information about how generating tuples $\overline{a}$ and $\overline{b} $ relate in ambient spaces.  \\

Approximate isometries can be induced by finite partial embeddings, i.e., partial functions $f:X \rightharpoonup Y$, with $\text{dom}(f) = X_0$ a finite set, which induce lattice embeddings $f:<X_0> \rightarrow Y$. More generally, for any $X_0 \subseteq X$ and a (not necessarily finite) partial embedding $f:X_0\rightarrow Y$, we let $\psi_f:X\times Y \rightarrow \R$ be defined by $\psi_f(x,y) = \inf_{z\in X_0} \|x-z\| + \|y- f(z)\|$.   Observe that if $x \in X_0$, then $\psi_f(x,y) = \|f(x) - y\|$.  If $X_0 \subseteq X$, we also have an approximate isometry $\psi_{Id_{X_0}}:X_0\times X \rightarrow \R$ with $Id_{X_0}$ the inclusion maps from $X_0$ to $X$, where $\psi_{Id_{X_0}}(x,y) = \| x - y\|$.\\

There is also a ``pseudoinverse'' operation: if $\psi(x,y)$ is an approximate isometry, we let $\psi^*(y,x) = \psi(x,y)$. Clearly $\psi^{**} = \psi$. We can also ``compose'' approximate isometries.   If $\phi:X\times Y \rightarrow \R$ and $\psi: Y\times Z \rightarrow \R$ are approximate isometries, then $\psi \phi: X\times Z \rightarrow \R$ with $\psi\phi(x,z)= \inf_{y\in Y}( \phi(x,y)+ \psi(y,z))$ is also an approximate isometry by \cite[Lemma 2.3(i)]{ben15}. For example, if  $f:A_0 \rightarrow C$ and $g:B_0 \rightarrow C$ generate embeddings from $<A_0>$ and $<B_0>$ to $C$ respectively, then the map $\psi_{g}^* \psi_{f}: <A_0 > \times <B_0> \rightarrow \R$, where
\[  \psi_g^* \psi_f(x,y) = \inf_{z\in C}( \psi_f(x,z)+ \psi_g(y,z)). \]

is also an approximate isometry. Note that if $A_0 = \overline{a}$, $B_0 = \overline{b}$, and  $d(f(\overline{a}), g(\overline{b}) )$ is small, then $\psi_g^* \psi_f(a_i, b_i)$ will also be small, and the converse holds true as well.  In fact, for $x \in A_0$ and $y\in B_0$, we have $\psi_{g}^*\psi_{f}(x, y) = \| f(x) - g(y) \|$ (here $A_0$ and $B_0$ need not be finite).  Thus we can see approximate isometries as marking conditions for the "strength" of a joint embedding.  An approximate isometry $\psi$ may originally be defined on some $X_0 \times Y_0$, with $X_0 \subseteq X$ and $Y_0 \subseteq Y$, but it can be extended to $X\times Y$ by the composition $\psi_{Id_{Y_0}} \psi \psi^*_{Id_{X_0}}$.  However, if the ambient spaces are clear from context, we will just write $\psi$ to refer to the extended approximate isometry.   Finally, composition and involution as described above work analogously together like the multiplication and inversion group operations.  In particular, composition is associative and $(\psi \phi)^* = \phi^* \psi^*$ (see \cite[Lemma 2.3(ii)]{ben15}).\\

We say that  $\psi$ is \textit{refined by}, or \textit{coarsens} $\phi$ if $\phi(x,y) \leq \psi(x,y)$ for all $(x,y) \in X\times Y$.  Given lattices $X$ and $Y$, we let $\mathcal {A}px(X,Y) \subseteq [0,\infty]^{X\times Y}$, equipped with the product topology on $[0,\infty]^{X\times Y}$, be the set of all approximate isometries  generated by finite partial embeddings between elements in $\mathcal K'$, composition, coarsening, and any point-wise limit of such maps.  For $\psi \in \mathcal Apx(X,Y)$, we let $\mathcal Apx^{<\psi}(X,Y)$ be the interior of the set of refinements of $\psi$.    If $\mathcal Apx^{<\psi}(X,Y) \neq \emptyset$, we say that $\psi$ is a \textit{strictly approximate isometry} and use the notation $\phi < \psi$ to mean $\phi \in \mathcal Apx^{<\psi}(X,Y)$.  Intuitively, strictly approximate isometries do not impose strong conditions on possible joint embeddings except on some finite set (see Lemma \cite[Lemma 3.8(ii) ]{ben06}), so they leave much room for refinement. While the set $\mathcal Apx(X,Y)$ seems complicated, by \cite[Lemma 3.8(iv)]{ben15}, it actually is comprised of the closure of coarsening and pointwise limits of approximate isometries in the form of $\psi^*_{g} \psi_{f}$ (extended to $X\times Y$) where $f$ and $g$ are finite partial embeddings.\\

 Suppose $\psi:X\times Y \rightarrow \R$ is an approximate isometry, and let $r > 0$.   We say that $\psi$ is $r$-total if $\psi^*\psi \leq \psi_{Id_X} +2r$.  It is not hard to show that if $f:X\rightarrow Y$ is an embedding, then $\psi_f^* \psi_f = \psi_{Id_X}$, so any such $\psi_f$ is $r$-total for all $r > 0$. \\

Approximate isometries can also be used to characterize the AP: for every $A,B\in \mathcal K'$, $<\overline{a}> \subseteq A$ and embedding $f:<\overline{a}> \rightarrow B$, there exist $C\in \mathcal K'$ and embeddings $g:A\rightarrow C$ and $h:B\rightarrow C$ such that $\psi_h^*\psi_g \leq \psi_f$ (with the necessary extensions on $f$).   See \cite[Definition 3.5(iii)]{ben06} for the generalized definition of the NAP using approximate isometries. Using the fact that $\psi_h^*\psi_g(x,y) = \| g(x) - h(y) \|$ and $\psi_f(x,y) = \| f(x) - y\|$, one can easily show that this definition is equivalent to our current working definition of the AP. Note also the inequality; this is due to the fact that $f$ is only defined on $<\overline{a}>$ while $g$ is defined on all of $A$, so the extension of $\psi_f$ to $A\times B$ contains less limiting information than  $ \psi_h^*\psi_g(x,y)$.  \\

 We are now ready to prove the following:
\begin{theorem}\label{t:finitely-branchable-fraisse}
	$\mathfrak{ BL}$ can be constructed as the limit of an increasing sequence of $\ell_\infty^m(\ell_1^n)$ lattices.  In particular, it is finitely branchable.
\end{theorem} 

\begin{proof}
We construct an increasing sequence of finite dimensional lattices $\mathcal A_n$ as in the proof of Lemma 3.17 in \cite{ben15} with $\mathfrak{ BL}$ isometric to $ \overline{\bigcup_n \mathcal A_n}$.  Let $\mathcal A_1 \in \mathcal K'$, and let $K_{n,0}$ be a countable dense subset of $\mathcal K'_n$.  Since $\mathcal K'$ is dense in $\mathcal K$, we have $K_{n,0}$  dense in $\mathcal K_n$. We proceed by induction. Suppose $\mathcal A_k$ has been defined for all $k\leq n$.  Suppose also that $A_{k,0} \subseteq \mathcal A_k$ is countable and dense in $\mathcal A_k$ for all $k\leq n$, with $A_{k,0}\subseteq A_{k+1,0}$ for each $k < n$.    \\

By \cite[Lemma 3.8(ii)]{ben15}, for any finite tuples $\overline{a}$ and  $\overline{b}$ we can ensure the existence of a countable set $C(\overline{a},\overline{b}) \subseteq \mathcal K'$  such that every $C\in C(\overline{a},\overline{b})$ contains an isometric copy of $<\overline{a}>$ and $<\overline{b}>$, and every strictly approximate isometry $\psi: \overline{b} \times \overline{a} \rightarrow \Q$  can be refined in $<\overline{b}> \times <\overline{a} >$ by some $\psi_{f}^*\psi_{g}$ with $f:<\overline{a} > \rightarrow C$ and $g:<\overline{b}> \rightarrow C$ for some $C \in C(\overline{a},\overline{b})$ (by \cite[Lemma 2.8(ii)]{{ben15}}, such strictly approximate isometries $\psi: \overline{b} \times \overline{a} \rightarrow \Q$ actually exist).  Let $$C_k = \bigcup_{\substack{ \overline{b} \in K_{n,0} \\ \overline{a} \subseteq A_{k,0} }} C(\overline{a} , \overline{b}).$$ 

To construct $\mathcal A_{n+1}$, take the first $n$ lattices $C_k^1,..., C_k^{n}$ in each $C_k$ for $k \leq n$, and amalgamate them one after another.  Here $ \overline{a_{i,j}} \subseteq A_{i,0}$ for some $i\leq n$,  and $<\overline{a_{i,j}}>$ and $<\overline{b_{i,j}} >$ both into $C_{i}^j \in C(\overline{a_{i,j}},\overline{b_{i,j}})  \subseteq C_i: $  
\begin{center}
	
\begin{tikzcd}[scale cd=0.8]
\mathcal A_n \arrow{rr} &  & * \dots \quad * \arrow{rr}{\iota'} & & * \quad  \dots \quad * \arrow{r} & \mathcal A_{n+1} \\
 <\overline{a_{1,1}}>\arrow{u}  \arrow{r} & C_1^1 \arrow{ru} &  <\overline{a_{i,j}}>\arrow{u} \arrow{r}{f} & C_i^j\arrow{ur}{h} &  \quad \dots \quad <\overline{a_{{n,n}}}> \arrow{u} \arrow{r} & C_{n}^n \arrow{u} \\
 <\overline{b_{1,1}}> \arrow{ru} & & <\overline{b_{i,j}}> \arrow{ru}{\iota} & &   <\overline{b_{n,n}}> \arrow{ru} &
\end{tikzcd}
\end{center}

  Note that in each case, $\mathcal A_k \in \mathcal K'$, so we can if necessary enlarge $\mathcal A_k$ and assume that $\mathcal A_k = \ell_\infty^{m_k}(\ell_1^{n_k})$ for some $m_k, n_k \in \N$.\\
  
  Observe that for each tuple $\overline{a} \subseteq \mathcal A_k$, $\overline{b} \in K_{k,0}$, and each strictly approximate isometry $\psi:\overline{b} \times \overline{a} \rightarrow \Q$, we have some $m > k$ such that $\overline{a} = \overline{a_{i,j}} $,  $\overline{b} = \overline{b_{i,j}} $ for some $i, j \leq m$ with $C_i^j \in C(\overline{a},\overline{b})$ and embeddings $f:<\overline{a_{i,j}} > \rightarrow C_i^j$ and $\iota:<\overline{b_{i,j}}> \rightarrow C_i^j$ such that $\psi^*_f \psi_\iota < \psi$. Additionally, there is an embedding $h:C_i^j \rightarrow \mathcal A_{m}$ with $ \psi_h^* \psi_{\iota'} \leq \psi_f$. In particular, we have 
  $$\psi_{h\iota}|_{<\overline{b_{i,j}}>\times <\overline{a_{i,j}}>} = \psi_{f}^* \psi_\iota:$$
  Indeed, given $x \in <\overline{a_{i,j}}>$ and $y \in <\overline{b_{i,j}}>$,
  \begin{align*}
  \psi_{h\iota}(y,x) & = \| h\iota(y) - x \|  = \| h\iota(y) - \iota'(x) \| \\
  					& = \| h\iota(y) - hf(x) \| = \| \iota(y) - f(x) \| = \psi_f^* \psi_\iota.		
  \end{align*}
  Thus $\psi_{h\iota} \leq \psi_f^* \psi_\iota	 < \psi$.  Furthermore, $\psi_{h\iota}$ is $r$-total on $\overline{b_{i,j}}$ for all $r > 0$, since $h\iota$ is an embedding.  Thus by \cite[Lemma 3.16]{ben15}, $\overline{\bigcup_n \mathcal A_n}$ is a Fra\"iss\'e limit for $\overline{\mathcal K'}$-structures, where $\overline{\mathcal K'}$ is the Fra\"iss\'e completion of $\mathcal K'$.   By Theorem \ref{l:incomplete-fraisse}, we have $\overline{\mathcal K'} = \mathcal K$. This implies that $\overline{\bigcup_n \mathcal A_n}$ is also a Fra\"iss\'e limit of $\mathcal K$, so by uniqueness, $\overline{\bigcup_n \mathcal A_n}$ is isometric to $\mathfrak{BL}$. \end{proof}

A lattice that is finitely branchable can be expressed as an inductive limit of finite dimensional lattices.  We also have the following: 

\begin{theorem}\label{t:fb-implies-fg}
	Any finitely branchable lattice can be generated by two elements.  
\end{theorem}

\begin{proof}
	
	Let $X$ be finitely branchable, and let $(x_\sigma)_{\sigma\in T}$ be a finitely branching tree densely spanning $X$.  Recall that as a tree, $T \subseteq \bigcup_M \prod_n^M A_n$, where each $A_n$ is a finite nonempty set. We will find two elements $u$ and $v$ such that $X = <u,v>$.  Let $u = x_{\emptyset}$ and let $S_n = \{ x_\sigma: |\sigma| = n \}$ as in the proof of Lemma \ref{l:incomplete-fraisse}.  Consider now $X_1 = \text{span}( S_1)$, and let $$v_1 = \sum_{|\sigma| = 1} a_\sigma x_\sigma, $$ where $0 < a_\sigma$ and the $a_\sigma$'s are mutually distinct. The mutual distinction enables each $x_\sigma$ to be produced using lattice operations over $u$ and $v_1$.  For example, take $a_{\rho} = \max a_\sigma$, and pick $c$ such that $c\alpha_\rho > 1$ but for all $\tau \neq \rho$, $ca_\tau < 1$.  Recall that $u = x_\emptyset = \sum_{|\sigma| = 1} x_\sigma$, so $(cv_1 - u) \vee 0 = (c\alpha_\rho - 1) x_\rho$.  We then make the same argument, but for $u - x_\rho$ and $v_1 - a_\rho x_\rho$, thus generating each  successive $x_\sigma$ for all $\sigma \in S_1$.\\
	
	 Suppose that for all $k\leq n$, $v_k$ has been selected and that for each $k$ we have a finite sequence of functions $(\phi_k^i(x,y))_i$ generated by lattice operations $+, \wedge, r\cdot$ (where $r$ is real), with corresponding moduli of continuity $\Delta_k^i:\R^+ \rightarrow ( 0,1]$.  Suppose we also have: %(we can assume each $\Delta_k^i$ is bounded by $1$)
	\begin{itemize}
		\item For each $k \leq n$, $v_k = \sum_{|\sigma| = k} a_\sigma x_\sigma$, with $a_\sigma > 0$ and mutually distinct.
		\item For each $k \leq n$, $<u,v_k> = \text{span}(S_k)$
		\item For each $k\leq n$, $(\phi_k^i(u, v_k))_i$ is a $2^{-k}$-net in the unit ball of $\text{ span}(S_k)$.
		\item For each $k< n$, $$ \| v_k - v_{k+1} \| < \frac{\min_{i,j\leq k}( \ (\Delta_j^i  (2^{-k}) )_i \ )}{2^{k}}$$

	\end{itemize}
	Note that  $$v_k = \sum_{|\sigma| = k} a_\sigma \sum_{m \in A_{k+1}} x_{\sigma ^\frown m},$$

	so for each $m\in A_{n+1}$ and $\sigma \in \prod_1^n A_k $, pick positive, mutually distinct $a_{\sigma^\frown m}$ such that 
	 $$ |a_\sigma - a_{\sigma^\frown m} | < \frac{\min_{i,j\leq n}( \ (\Delta_j^i  (2^{-n}) )_i \ )}{2^{n} |S_{n+1}|}. $$
	 Thus if $v_{n+1} :=  \sum_{|\sigma| = n+1} a_\sigma x_\sigma$, we have $\| v_n - v_{n+1} \| < \frac{\min_{i,j\leq n}( \ ( \Delta_j^i (2^{-n}))_i \ )}{2^{n}}$.  Now $<u,v_{n+1}> = \text{ span} (S_{n+1})$.  For each  $\sigma $ with $|\sigma| = n+1$, pick $0<s < a_\sigma < r$ that for all $\tau \neq \sigma$ with $\tau \in \prod_1^{n+1}A_k$, either $\tau > r $ or $\tau < s$. Let $x = (v_{n+1} - su)_+$ and $y = (v_{n+1} - ru)_+$.  Then for some large enough $C$, $(Cy - x)_+$ is a multiple of $x_\sigma$. Finally, let $(\phi_{n+1}^i(x,y))_i$ be a finite collection of functions generated by lattice operations such that $(\phi_{n+1}^i(u,v_{n+1}) )_i$ is a $2^{-n-1}$-net in the unit ball of $\text{ span}(S_{n+1})$.  Let $v = \lim v_n$.\\
	
	We show that $<u,v> = X$. Observe that the set $\{ \phi_n^i(u, v_n): \ n\in \N \}$ is dense in $X$ by the above properties. Let $\varepsilon > 0$, and pick $n$ such that $2^{-n} < \varepsilon$. Then  $\|v_n - v_{n+1} \| < \frac{\Delta_n^i  (\varepsilon)  }{2^{n+1}}$ and $\phi_n^i(u,v_n)$ is $\varepsilon$-dense in $\mathbf B(\text{span}(S_n))$.  Furthermore, for all $m > n$, we have $\| v_n - v_m \| <  \min_i \Delta_n^i  (\varepsilon) \sum_{j = n+1}^{m} 2^{-j}$. Thus  $\| v_n - v \| \leq \Delta_n^i  (\varepsilon)$, so $\| \phi_n^i(u,v_n) - \phi_n^i(u,v) \| < \varepsilon$ for all $\phi_n^i$.  This implies that the set $\{ \phi_n^i(u,v) | \ n\in \N \}$ is dense in $X$, so we are done. \end{proof}

	%First assume that $X$ is atomless.  Let $X$ be finitely branchable, and let $(x_\sigma)_{\sigma \in T}$ be a linearly dense spanning tree, with $x_\emptyset= u$ a weak unit.  Then there exists an injective contractive lattice homomorphism $\phi$ from a $C(K)$ space onto $\cup_n [-nu, nu]$, which is dense in $X$, with $\mathbf{1} \mapsto u$.\\

%For all $n\in N$, $ u= \sum_{|\sigma| = n} x_\sigma$.  Then each $x_\sigma = \phi(\mathbf{1}_{K_\sigma})$ for some clopen $K_\sigma \subseteq K$.  Let $T' \subseteq T$ be the pruned subtree consisting of all $\sigma \in T$ such that $x_\sigma \neq 0$. Then the set $[T']$ of infinite branches in $T'$ is itself a compact Polish perfect zero-dimensional space, so by Theorem 7.4 in \cite{kech95}, $[T']$ is homeomorphic to $\Delta$, and there exists a continuous open surjection from $K$ to $\Delta$ mapping clopen sets to clopen sets.  This induces an isometric embedding $C(\Delta)) \hookrightarrow C(K)$, mapping indicator functions to indicator functions such that $\text{span}((x_\sigma)_{\sigma\in T}) \subseteq \phi(C(\Delta)) $.  Let $e_1 = u$ and $e_2 = \phi(f)$, where $f$ is the coordinate function on $\Delta$.  Now $VL(1_\Delta, f)$ is dense in $C(\Delta)$ so $VL(u, \phi(f))$ is dense in $\cup_n [-nu, nu]$, and thus is dense in $X$. \\

Theorem \ref{t:finitely-branchable-fraisse} combined with Theorem \ref{t:fb-implies-fg} yields a surprising result:
\begin{corollary}
 The lattice $\mathfrak{ BL}$ is finitely generated.
\end{corollary}

\begin{remark}
	 Finite generation implies that $\mathfrak{ BL}$ is not stably homogeneous in the sense defined by Lupini in \cite{lup16}. That is, given finitely generated $A$ and embeddings $f:A \rightarrow \mathfrak{ BL}$ and $g:A \rightarrow \mathfrak{ BL}$, we cannot guarantee that for all $\varepsilon > 0$, there is some automorphism $\phi$ on $\mathfrak {BL}$ such that $\| \phi \circ f - g \| < \varepsilon$.  Thus approximating over a finite number of elements rather than by norms is the best, in some sense, that can be done in terms of homogeneity. \\
	 
	 Suppose otherwise.  Since $\mathfrak{ BL} $ can be generated by two elements $x_1, x_2$, we consider $e \in \mathbf S(\mathfrak{ BL})_+$ and embedding $f: <x_1,x_2> \rightarrow \mathfrak{ BL}$ such that the image does not have full support and $f(e)$ is disjoint from $e$ (we can do this, for example, by finding a copy of $\mathfrak{ BL} \oplus_\infty \R$  that is in $\mathfrak{BL},$ and pick $e\in \mathbf S(\R)_+)$.  If $\mathfrak{ BL}$ were stably homogeneous, there would exist a lattice automorphism $\phi$ on $\mathfrak {BL}$ such that $  \| f - \phi \| < \frac{1}{2}$.  Then $\|f (\phi^{-1}(e)) - e \| = \|(f -\phi)(\phi^{-1}(e)) \|< 1/2$, but $f(\phi^{-1}(e)) $ is disjoint from $e$, which means $\|f (\phi^{-1}(e)) - e \| \geq 1$, a contradiction.
\end{remark}

\section{An alternate construction of $\mathfrak{ BL}$ and some of its properties } \label{s:sep-univ-dist}
  The Fra\"iss\'e limit $\mathfrak{ BL}$ is clearly of approximately universal disposition both for finite dimensional lattices and for finitely generated lattices.  In this section, we show that separable lattices of approximately universal disposition for finitely generated lattices are isometric to $\mathfrak{BL}$.  In addition, finitely branchable lattices of approximately universal disposition for finite dimensional lattices are isometric to $\mathfrak{BL}$. A bi-product of the latter is a simplified construction of $\mathfrak{ BL}$ which allows us to explore some of its structural properties.\\

We first show that approximate universal disposition can be broadened to include extensions of $\varepsilon$-isometries:

\begin{lemma}\label{l:disposition-to-fraisse}
	Suppose $X$ is of approximately universal disposition for finitely generated Banach lattices, let $< \overline{a}> = A\subseteq X$ and $B$ be finitely generated, and let $f:A \rightarrow B$ be a $(1+\varepsilon')$-embedding.  Then for all $\varepsilon > \varepsilon'$ and for all $\delta > 0$, there exists a $(1+\delta)$-embedding $g:B\rightarrow X$ such that $\| g\circ f( \overline{a}) -  \overline{a} \| < \varepsilon$.  
\end{lemma}

\begin{proof}
	By Theorem \ref{t:two-embeddings-one-isomorphism} there is a lattice $Z$ and embeddings $h_1:A \rightarrow Z$ and $h_2: B\rightarrow Z$ such that $\|h_2 \circ f - h_1\| \leq \varepsilon'$. We can assume that $Z$ is finitely generated as well, since we can embed into $\mathcal U$ if necessary.  Decreasing $\delta$ as necessary, we can suppose that $(1+\delta)\varepsilon' + \delta < \varepsilon$.  Then there exists a $(1+\delta)$-embedding $g':Z\rightarrow X$ such that $\| g' h_1 ( \overline{a}) -  \overline{a} \| < \delta$. Then 
	$$\|g' h_2f( \overline{a}) -  \overline{a} \| \leq \|g'\| \|h_2( \overline{a}) - h_1 ( \overline{a}) \| + \|g'h_1 ( \overline{a}) -  \overline{a} \| < (1+\delta)\varepsilon'+ \delta <\varepsilon.$$
	Let $g = g' \circ h_2$, and we are done. \end{proof}

We can now show the following: 

\begin{theorem}\label{t:1e-isometric-unique}
	Any separable Banach lattice of approximately universal disposition for finitely generated lattices is isometric to $\mathfrak{ BL}$.
\end{theorem}

\begin{proof}
	
	The proof follows that of Theorem 1.1 in \cite{kubis13}. Suppose $X$ and $Y$ are lattices of approximately universal disposition for finitely generated lattices. We will then construct a lattice isometry.  Let $(x_n)$, $(y_n)$ be dense in $X$ and $Y$, with $x_1 = 0$ and $y_1 \geq 0$.  Given $\varepsilon > 0$, let $\varepsilon_n \downarrow 0$ be a decreasing sequence such that $\varepsilon_n < 2^{-n-1}$. Throughout, we let $ \overline{x_n} = (x_1,...,x_n)$, and let $X_n = < \overline{x_n} >$, with the same notation for $ \overline{y_n}$ and $Y_n$.  Finally, let $f_1: X_1 \rightarrow Y_1$ be the trivial isometry. \\
	
	We begin our construction:  let $g_1:Y_1 \rightarrow X$ be a  $(1+\varepsilon_1)$-isometry.  Note this isometry exists. Now take $\tilde{X_2} = BL(X_2 \cup g_1(Y_1))$.  This lattice is also finitely generated by $\tilde{x_2} =  \overline{x_2} \cup g_1( \overline{y_1})$, so we have the map $g_1:Y_1 \rightarrow \tilde{X_2}$, and by Lemma \ref{l:disposition-to-fraisse}, pick a $(1+\varepsilon_2)$-embedding $f_2: \tilde{X_2} \rightarrow Y$ such that $d(f_2 g_1( \overline{y_1}),  \overline{y_1 }) < \frac{1}{2^2}$.  Now let $\tilde{Y_2} = BL(f_2(\tilde{X_2})  \cup Y_2)$.   Use Lemma \ref{l:disposition-to-fraisse} again to generate a $(1+\varepsilon_3)$- embedding $g_2: \tilde{Y_2} \rightarrow X$ such that $d( g_2 f_2( \tilde{x_2}) ,\tilde{x_2} )< \frac{1}{2^3}$. \\
	
	We can proceed inductively by constructing finitely generated subspaces $\tilde{X}_n = BL(X_n \cup g_{n-1}(\tilde{Y}_{n-1}) )$ and $\tilde{Y}_n = BL(Y_n \cup f_{n}(\tilde{X}_{n}) ) $ with corresponding tuples $\tilde{x_n}  = \overline{x_n} \cup g_{n-1}(\overline{y_{n-1}})$ and $\tilde{y_n}  = \overline{y_n} \cup f_{n}(\overline{y_{n}})$, as well as $ (1+\varepsilon_{2(n-1)}) $-embeddings $f_n: \tilde{X}_n \rightarrow \tilde{Y}_{n}$ and $(1+\varepsilon_{2n-1}) $-embeddings  $g_n: \tilde{Y}_n \rightarrow \tilde{X}_{n+1}$ such that $d(\tilde{y_n}, f_{n+1} g_n(\tilde{y}_n)) < \frac{1}{2^{2(n-1)}} $ and $d(\tilde{x_n}, g_n f_n(\tilde{x}_n)) < \frac{1}{2^{2n-1}} $.  Note that for each $k\leq n$,  we have 
	\begin{align*}
	&	\| f_{n+1}(x_k) - f_{n}(x_k) \|\\
	 = & \| f_{n+1}(x_k - g_nf_n(x_k) + g_nf_n(x_k)) - f_{n+1}(x_k) \| \\
 \leq & \| f_{n+1}(x_k - g_nf_n(x_k) ) \| + \| (f_{n+1} g_n - Id) f_{n}(x_k) \|  \\
\leq	& \frac{2}{2^{2(n-1)}}  + \frac{2}{2^{2(n-1)}} = \frac{1}{2^{2(n-2)}}
	\end{align*}
	
	so the sequence $f_n(x_k)_{n \geq k}$ is Cauchy.  The same is true for $g_n(y_k)_{n \geq k}$.  Let $f = \lim f_n$, and let $g = \lim g_n$. These exist, since $(x_k)$ and $y_k$ are dense in $X$ and $Y$.  Furthermore, $f$ and $g$ are inverses of each other, and they are each isometries. \end{proof}

\begin{comment}
Recall that proximity in the Fra\"iss\'e metric does not indicate the existence of a lattice isomorphism.  Indeed, by Lemma \ref{l:incomplete-fraisse}, any separable finitely generated lattice is arbitrarily close to a finite dimensional lattice in $\mathcal K'$, so the construction of $\mathfrak{ BL} $ using the same argument as in Section \ref{s:univ-dist} with countably finitely generated lattices will not necessarily yield a lattice of approximately universal disposition for finitely generated lattices.  However, with some adjustments, w
\end{comment}
We now construct a separable lattice of approximately universal disposition for finite dimensional lattices.  The approach is a modification of that in \cite[Section 5]{avi11} for the Gurarij space. \\

 Let $\mathfrak J$ be the collection of embeddings between finite dimensional lattices in $\mathcal K$. Since any such lattice isometrically embeds into $\mathcal U$, we can assume that $\mathfrak J$ is a set by limiting it to embeddings between finite dimensional sublattices of $\mathcal U$.   Let $\mathfrak J_0$ be a countable dense subset of $\mathfrak J$ in the following sense: for all embeddings $f:A \rightarrow B$ with $B$ finite dimensional and for all $\varepsilon > 0$, there exists $u:A' \rightarrow B' \in \mathfrak J_0$, and $(1+\varepsilon)$-isometries $\iota_A:A \rightarrow A'$ and $\iota_B:B\rightarrow B'$ such that $u\circ \iota_A = \iota_B \circ f$. In addition, for a separable lattice $X$, let $\mathfrak L(X)$ be the set of all maps $ v:A'\rightarrow X$ which are $C$-embeddings for some $C\geq 1$ with $A'\in \text{ Dom}(\mathfrak J_0)$.  Let $\mathfrak L_0$ be a countable subset of $ \mathfrak L(X)$ which is dense in the following sense: for all $\varepsilon > \varepsilon' > 0$ and $(1+\varepsilon')$-embeddings $f:A' \rightarrow X$ with $A' \in \text{Dom}(\mathfrak J_0)$, $\mathfrak L_0$ contains an $(1+\varepsilon)$-embedding $v:A' \rightarrow X$ such that $\| v - f \| < \varepsilon$.
  \\
 
Let $ X_{0,0} = X$, and suppose now that $X_{n,k}$ has been constructed.  Let $\mathfrak L_n$ be a countable subset of $ \mathfrak L(X_{n,0})$ which is dense in the manner described for $\mathfrak L_0 $ and $\mathfrak L(X)$.  Finally, let $\Gamma_n = \{ (u,v) \in \mathfrak J_0 \times \mathfrak L_n: \text{dom}(u) = \text{dom}(v) \}$, and let $((u^n_i,v^n_i))_i$ be an enumeration of $\Gamma_n$. We then construct $X_{n, k+1}$ by amalgamating as follows.  Given $(u^n_k,v^n_k) \in \Gamma_n$ with $v^n_k$ a $C$-embedding, we use part 2 of Corollary \ref{c:morph-to-met-amalg} to get an embedding $\iota:X_{n,k} \rightarrow X_{n,k+1}$ and $C$-embedding $w: \text{cod}(u^n_k) \rightarrow X_{n,k+1}$ such that the following diagram commutes:
\begin{center}
	\begin{tikzcd}
	\text{dom}(u^n_k) \arrow{r}{u^n_k} \arrow{d}{v^n_k} & \text{cod}(u^n_k) \arrow{d}{w} \\
	X_{n,k} \arrow{r}{\iota} & X_{n, k+1}
	\end{tikzcd}
\end{center}

Finally, we let $X_{n+1, 0} =\overline{ \bigcup_{k\in \N} X_{n,k}}$, and then let $X_{\omega_0}(X) = \overline{\bigcup_{n\in N} X_{n,0}}$.
\begin{theorem}\label{t:finite-disposition}
	 $X_{\omega_0}(X)$ is of approximately universal disposition for finite dimensional lattices.
\end{theorem}
\begin{proof}
Let $f:A\rightarrow B$ and $g:A \rightarrow X_{\omega_0}(X)$ be embeddings, with $A$ and $B$ finite dimensional.  Given $\varepsilon > 0$, by density of embeddings and spaces in $\mathfrak J_0$ and $\mathfrak L_n$, pick an embedding $u:A'\rightarrow B'$ with $u\in \mathfrak J_0$  such that there are  $(1+\varepsilon/2)$-isometries $i_1:A\rightarrow A'$ and   $i_2:B\rightarrow B'$ with $u\circ i_1 = i_2 \circ f$. Since $g\circ i_1^{-1}$ is also a $(1+\varepsilon/2)$-embedding, there exists $n$ and a $(1+\varepsilon)$-embedding $v:A' \rightarrow X_{n,0}$ such that $v\in \mathfrak L_n$ and $\| v - g \circ i_1^{-1}\| < \varepsilon$. Then $\| v\circ i_1 - g \| < (1+\varepsilon) \varepsilon$.  Using the construction above, we thus have the following commutative diagram:
\begin{center}
	\begin{tikzcd}
A \arrow{r}{f} \arrow{d}{i_1} & B \arrow{d}{i_2} \\
A' \arrow[r]{}{u} \arrow[d]{}{v}    &  B' \arrow{d}{\iota_{B'}} \\
X_{n, 0} \arrow{r}{\iota} & X_{n+1,0}
\end{tikzcd}
\end{center}
In the diagram,  $i_1$, $i_2$, $v$ and $\iota_{B'}$ are each $(1+\varepsilon)$-embeddings. We now let $ h:B \rightarrow X_{\omega_0}(X) = \iota_{B'} \circ i_2$.  This is clearly a $(1+\varepsilon)^2$-embedding.  Finally, for all $x\in \mathbf{B}(A)$, we have 
\begin{align*}
\|g(x) - hf(x)\| &= \|g(x)  -  \iota_{B'}i_2 f(x) \| \\
					& = \|g(x) - vi_1(x) \| < (1+\varepsilon)\varepsilon.  
\end{align*}

Thus $X_{\omega_0}(X)$ is of approximately universal disposition for finite dimensional lattices. \end{proof}

One can use a similar argument to construct (non-separable) lattices of universal disposition by amalgamating over all combinations of embeddings of separable spaces $\omega_1$ times rather than selecting a countable subset each step (see \cite[Theorem 5.3]{AvilesTradecete2020}).\\

We can adapt our construction with additional conditions as a way to discern the structure of $\mathfrak{BL}$. For instance, we can start with $X \in \mathcal K'$ in particular, and then inductively construct increasing lattices $X = X_0 \subseteq X_1 \subseteq... \subseteq X_n \subseteq X_{n+1} \subseteq ....$ with each $X_n \in \mathcal K'$.  First, we ensure that $\mathfrak J_0$ and each $\mathfrak L_n$ consist only of maps between lattices in $\mathcal K'$. This is possible assuming $X_{0}$ is in $\mathcal K'$ and by Lemma \ref{l:fd-into-1infty-l1}.  Then given $X_n \in \mathcal K'$, we can construct $X_{n+1}\in \mathcal K'$ by applying parts 2 and 3 of Corollary \ref{c:morph-to-met-amalg} over the $n^{th}$ pair $(u^k_n,g^k_n) \in \Gamma_k$ for each $k < n$, followed by the first $n$ pairs in $\Gamma_n$.  

\begin{center}
	\begin{tikzcd}[scale cd=0.65]
	\text{dom}(u_n^1) \arrow{r}{u_n^1} \arrow{d}{v_n^1} & \text{cod}(u_n^1) \arrow{dr} & \text{dom}(u_n^2) \arrow{r}{u_n^2} \arrow{d}{v_n^2} & \text{cod}(u_n^2) \arrow{dr} & \dots & \text{dom}(u_n^n) \arrow{r}{u_n^n} \arrow{d}{v_n^n} & \text{cod}(u_n^n) \arrow{d} &\\
	X_n \arrow{rr} & & * \arrow{rr} & & * \arrow{r}[description]{\dots} &  * \arrow{r} & X_{n+1}
	\end{tikzcd}
\end{center}

   Thus for any pair $(u_n^k, v_n^k) \in \Gamma_k$ with $v_n^k$ a $C$-embedding for some $C$, there exists some $m > k$ (here $m =\max(k+1, n+1)$ ) and a $C$-embedding $\iota:\text{cod}({u_n^k})\rightarrow X_m$ such that $ v_{n}^k = \iota \circ u_n^k$. The lattice $\overline{ \bigcup_{n} X_{n}}$ is then a limit of finite dimensional lattices and is thus finitely branchable, and a small variation of the argument in Theorem \ref{t:finite-disposition} can be used to show that it is also of approximately universal disposition for finite dimensional lattices.  It turns out, however, that we have derived an alternate, simplified construction of $\mathfrak{ BL}$:

\begin{theorem} \label{t:finite-branch-unique}
	Any two finitely branchable lattices of approximately universal disposition for finite dimensional lattices are isometric. In particular, they are isometric to $\mathfrak{ BL}$ and are thus of approximately universal disposition for finitely generated lattices.
\end{theorem}
\begin{proof}
Suppose $X$ and $Y$ are two finitely branchable separable lattices of approximately universal disposition for finite dimensional lattices. As in the proof of Theorem \ref{t:1e-isometric-unique}, we simply construct an isometry $f:X\rightarrow Y$ with its inverse $g:Y\rightarrow X$.\\

Let $(X_n)$ and $(Y_n)$ be sequences of finite dimensional lattices generated by corresponding spanning trees, (here we let $X_n = \text{span}(\{ x_{\sigma}: |\sigma|= n \})$ such that $X = \overline{\bigcup X_n}$ and $Y = \overline{\bigcup Y_n}$. Let $\varepsilon_n \downarrow 0$ be a sequence such that $\prod (1+\varepsilon_n) < 1+\varepsilon$. We then proceed just like in the proof of Theorem \ref{t:1e-isometric-unique}, but with a modification.  Let $f_1:X_0 \rightarrow Y_0$ be an isometry (this is possible because $X_0$ and $Y_0$ are simply 1-dimensional lattices spanned by $x_\emptyset$ and $y_\emptyset$, respectively). Now, let $g_1:Y_1 \rightarrow X$ be a $(1+\varepsilon_1)$-embedding such that $\|x_\emptyset - g_1f_1(x_\emptyset) \| \leq \frac{1}{2}$,  By density of $\bigcup X_n$, and since $Y_1$ is finite dimensional, we can in fact ensure that $g_1$ maps into some $X_{k_2}$ for some $k_2\in \N$.  \\

 Rather than generating lattices $\tilde{X}_n$ and $\tilde{Y}_n$, using Lemmas \ref{l:disposition-to-fraisse} and \ref{l:finite-lattice-approximation}, we can pick $(1+\varepsilon_{2(n-1)})$-embeddings $f_n: X_{k_n} \rightarrow Y_{k'_n}$ and $(1+\varepsilon_{2n-1})$-embeddings $g_n: Y_{k'_n} \rightarrow X_{k_{n+1}}$ such that  $\|g_nf_n - Id_X \| < \frac{1}{2^{2n-1}}$, and similarly, $\| f_ng_{n-1} - Id_Y \| \leq \frac{1}{2^{2(n-1)}}$. Then 
 
 	\begin{align*}
 	\| f_{n+1} - f_{n} \|= & \| f_{n+1} - f_{n+1}g_nf_n + f_{n+1}g_nf_n - f_{n} \| \\
 \leq & \| f_{n+1}\| \| Id_{X}- g_nf_n \| + \| f_{n+1} g_n - Id_Y \| \|f_{n} \|  \\
 \leq	& \frac{2}{2^{2(n-1)}}  + \frac{2}{2^{2(n-1)}} = \frac{1}{2^{2(n-2)}}
 \end{align*}
 
 A similar argument can be made for the $g_n$'s. Thus for all $n$ and for all $x \in X_{k_n}$ the sequence $(f_m(x))_{m > n}$ is Cauchy.  The same is true for any $y \in Y_{k'_n}$ and sequence $(g_m(y))_{m>n}$.  Let $f = \lim f_n$ and $g = \lim g_n$, and we are done.\end{proof}

The assumption of finite branchability is essential in the above proof.  It is currently unknown, however, if there are lattices which are not finitely branchable but are of approximately universal disposition for finite dimensional lattices.\\

Since $\mathfrak{ BL}$ is finitely branchable, it contains many non-trivial projection bands.  Recall that a ideal sublattice $B \subseteq X$ is a\textit{ band }if for all $x\in X$ and sets $A \subseteq B$, if $x = \sup A,$ then $x\in B$.  Given a set $A\subseteq X  $, we let $A^\perp = \{ x\in X: x \perp a \text{ for all } a\in A \}$. $A^\perp$ is itself a band, and if $B$ is a band, then $B^{\perp \perp} = B$ (see \cite[Theorem 1.28]{abram}).  A band $B$ is a \textit{ projection band} if $X = B \oplus B^{\perp}$; that is, every $x\in X$ can be uniquely written as $x_1 +x_2$ with $x_1 \in B$ and $x_2 \in B^{\perp}$.  Note that if $B$ is a projection band, it induces a lattice projection $P:X\rightarrow B$, that is, a contractive lattice homomorphism onto $B$ with $P^2 = P$ and in particular, $P|_B = Id |_B$. Let $(x_\sigma)_{\sigma \in T}$ be a linearly dense spanning tree in $\mathfrak{ BL}$. Then it is clear that $\mathfrak{BL} = \oplus_{|\sigma| = n} \overline{\text{span}}(\{x_\tau:\tau \supseteq \sigma\})$, and that each sublattice $\overline{\text{span}} (\{ x_\tau: \tau \supseteq \sigma \})$ is a projection band. We thus have the following:

\begin{theorem}\label{t:bands-are-also-BL}
	Every non-trivial projection band in $\mathcal B\subseteq \mathfrak{ BL}$ is itself isometric to $\mathfrak{BL}$.
\end{theorem}
 \begin{proof}
 	First note that $\mathcal B$ itself is finitely branchable. Given a finitely branching tree $(x_\sigma)_{\sigma \in T} \subseteq \mathfrak{BL}$ and lattice projection $P:\mathfrak{ BL}\rightarrow \mathcal B$, we get $(P(x_\sigma))_{\sigma \in T}$ as a spanning tree for $\mathcal B$. \\
 	
 	 We now show that $\mathcal B$ is of approximately universal disposition for finite dimensional lattices.  By Theorem \ref{t:finite-branch-unique}, this implies that $\mathcal B$ must in fact be isometric to $\mathfrak{BL}$. \\
 	
  	 Let $\mathfrak J_0$,  $\Gamma_n$, $\mathfrak L_n$,  and $X_n$ denote the sets and lattices used in the construction preceding Theorem \ref{t:finite-branch-unique}.  Let $A \subseteq \mathcal B$, and $f:A\rightarrow B$ be an isomeric embedding between finite dimensional $A$ and $B$. Given $\varepsilon > 0$, the goal is to construct a $(1+\varepsilon)$-embedding $g:B\rightarrow \mathcal B$ such that $\|g\circ f - Id|_A \| < \varepsilon$. \\
  	 
  	 We begin with the case where $f(A)$ fully supports $B$.  Let $(a_i)_i$ and $(b_i)_i$ be finite sequences enumerating the atoms of $A$ and $B$, and let $N$ be such that $Nf\big(\sum_i a_i\big) \geq \sum_i b_i$.  Finally, let $x_{\mathcal B}$ be a weak unit in $\mathcal B$ such that $\sum_i a_i \leq x_{\mathcal B}$.  Then given $\delta > 0$, by Lemma \ref{l:finite-lattice-approximation} and density of $\cup X_n$ in $\mathfrak {BL}$, there exist $m \in \N$, a $(1+\delta)$-embedding $v:A'\rightarrow X_m$ with $v \in \mathfrak L_m$, $u:A'\rightarrow B'$ with $u \in \mathfrak J_0$,  $j_1: A \rightarrow A'$, $j_2: B \rightarrow B'$, and $x'_B \in X_m$ such that 
 	 
 	 \begin{enumerate}
 	 	\item For all $i$, $\|  Na_i - Nvj_1(a_i) \| < \frac{\delta}{\dim A}$ 
 	 	\item $j_1$ and $j_2$ are $(1+\delta)$- isometries with $u \circ j_1 = j_2 \circ f$
 	 	\item $\|Nx_{\mathcal B} - Nx'_B\| < \delta$ and $x_B' \geq v j_1\big(\sum_i a_i\big)$.
 	 \end{enumerate}
	
If necessary, we can replace $x_{\mathcal B}$ with $x_{\mathcal B} \vee P(x_B')$.  All prior conditions will still be fulfilled, so we can assume $x_{\mathcal B} \geq P(x'_B)$.   Now $(u,v) \in \Gamma_m$, so there exists some $n > m$ such that $B'$ also $(1+\delta)$-embeds into $X_n$: 
 	
 \begin{center}
 	\begin{tikzcd}
 	A \arrow{r}{f} \arrow{d}{j_1} & B \arrow{d}{j_2} \\
 	{A'} \arrow{r}{u} \arrow{d}{v}    &  B' \arrow{d}{\iota_{B'}} \\
 	X_m \arrow{r}{\iota} & X_n
 	\end{tikzcd}
 \end{center}

Furthermore, \begin{align*}
(*) \ Nx'_B &\geq N\sum_i v j_1(a_i) = N\sum_i \iota v j_1(a_i) =  N\sum_i \iota_{B'} u j_1(a_i) \\
& = \sum_i \iota_{B'}\iota_2\big( N\sum_i f(a_i)\big) \geq \iota_{B'}\iota_2 \big( \sum_i b_i \big).
\end{align*}.

Let $g:B\rightarrow \mathcal B$ be defined by $g = P\circ \iota_{B'} \circ j_2$. By $(*)$, we have that $Nx_{\mathcal B} \geq NP(x'_B) \geq g\big(\sum_i b_i\big)$. 
%$$(\sum i_{B'}j_2(b_i))\wedge x_B = \sum( i_{B'}j_2(b_i)\wedge x_B), $$ 
Furthermore, since $P$ is a band projection, it follows that for all $\sum_i c_i \iota_{B'}j_2 (b_i) $ with $0 \leq c_i \leq 1$,
$$Nx_{\mathcal B}\wedge \big(\iota_{B'}j_2 \sum_i c_i b_i\big) = Nx_{\mathcal B}\wedge P\iota_{B'}j_2\big(\sum_i c_ib_i\big)  = g\big(\sum_i c_ib_i\big).$$
 Therefore 
\begin{align*}
&\bigg\| \sum_i c_i \iota_{B'}j_2(b_i) - \sum_i c_ig(b_i) \bigg\|  = \\
&\bigg\| Nx'_{B}\wedge \sum_i c_i\iota_{B'}j_2(b_i) - Nx_{\mathcal B} \wedge \sum_i c_i\iota_{B'}j_2(b_i) \bigg\| < \delta,
\end{align*} 
%Thus  $ \frac{1-\delta}{1+\delta} \leq \| \sum_i c_i(b_i\wedge(Nx_{B})) \| \leq \frac{1+\delta}{1-\delta}$, 
so $\|\iota_{B'}\circ j_2 - g\| < \delta,$ and since $\iota_{B'}\circ j_2$ is a $(1+\delta)^2$-embedding, $g$ is a $\frac{(1+\delta)^3}{1-\delta}$-embedding. Finally, since the diagram commutes, we have $gf(a_i) = P\iota_{B'}j_2f(a_i) = Pvj_1(a_i)$. Since $a_i \in \mathcal B$, $P(a_i) = a_i$, so for all $\sum_i c_i a_i \in \mathbf S(A)$, by condition 1 we have
\begin{align*}
&\bigg\|gf\big(\sum_i c_ia_i\big) -\sum_i c_ia_i \bigg\|  = \\
&\bigg\|\sum_i c_i \big( Pvj_1(a_i) -  P(a_i)\big) \bigg\| \leq \sum_i c_i \|  vj_1(a_i) - a_i \|< \delta.
\end{align*}
Now $\delta$ can be arbitrarily small, so assume that $\frac{(1+\delta)^3}{1-\delta} < 1+\varepsilon$.  Since $g = P\circ \iota_{B'} \circ j_2$, it sends elements of $B$ into $\mathcal B$ thanks to composition by $P$. Thus $g$ satisfies the requirements.\\

Suppose now that $B$ is not fully supported by $A$. For all $\delta > 0$, we can perturb $f$ with a $(1+\delta)$-embedding $f': A \rightarrow B$ such that $f'(A)$ fully supports $B$ and $\|f - f' \| < \delta$.  By Lemma \ref{t:isomorphism-to-isometry}, let $B'$ be a copy of $B$ with a $(1+\delta)$-equivalent renorming so that $f':A\rightarrow B'$ is an embedding.

Now use the result above to get a $(1+\delta)$ embedding $g:B'\rightarrow  \mathcal B$ such that $\| Id|_A - g\circ f' \| < \delta.  $  

\begin{center}
	\begin{tikzcd}
	A \arrow{r}{f} \arrow{dr}{f'} \arrow{dd}{Id|_A} & B \arrow{d}{Id_B} \\
	& B'\arrow{ld}{g} \\
	\mathcal B
	\end{tikzcd}
\end{center}

Now with the original norm on $B$, $g$ is a $(1+\delta)^2$-embedding, and under the norm on $B'$, $\|f - f'\|_{B'} < (1+\delta)\delta. $ Then
$$\| g\circ f - Id \| \leq \|g\circ(f - f') \| + \|g\circ f' -Id\| \leq  \|g\| \|f-f'\|_{B'} + \delta < (1+\delta)^2 \delta + \delta. $$

We can then let $\delta$ be arbitrarily small.  Thus $\mathcal B$ is of approximately universal disposition for finite dimensional lattices. \end{proof}

%\bibliographystyle{abbrv}
%\bibliography{../bigfatbib}

\end{document}